\documentclass[a4paper,11pt]{amsart}

\usepackage{amsmath, amsthm, amsfonts, amssymb}
\usepackage[colorlinks=true, linkcolor=blue, citecolor=red, menucolor=black]{hyperref}
\usepackage{enumerate}

\usepackage{color}

\numberwithin{equation}{section}

\newtheorem{letterthm}{Theorem}

\newtheorem{letterapp}{Application}

\newtheorem*{corollary}{Corollary}

\newtheorem{thm}{Theorem}[section]
\newtheorem{lem}[thm]{Lemma}

\newtheorem{prop}[thm]{Proposition}

\theoremstyle{definition}

\newtheorem{df}[thm]{Definition}

\newtheorem{claim}[thm]{Claim}

\newcommand{\R}{\mathbf{R}}
\newcommand{\C}{\mathbf{C}}
\newcommand{\Z}{\mathbf{Z}}
\newcommand{\F}{\mathbf{F}}

\newcommand{\N}{\mathbf{N}}

\newcommand{\cX}{\mathcal{X}}

\newcommand{\cU}{\mathcal{U}}

\newcommand{\Ad}{\operatorname{Ad}}
\newcommand{\id}{\text{\rm id}}

\newcommand{\Aut}{\mathord{\text{\rm Aut}}}
\newcommand{\rL}{\mathord{\text{\rm L}}}
\newcommand{\rB}{\mathord{\text{\rm B}}}
\newcommand{\rC}{\mathord{\text{\rm C}}}
\newcommand{\dom}{\mathord{\text{\rm dom}}}
\newcommand{\wm}{\mathord{\text{\rm wm}}}
\newcommand{\an}{\mathord{\text{\rm an}}}
\newcommand{\ap}{\mathord{\text{\rm ap}}}

\newcommand{\rT}{\mathord{\text{\rm T}}}

\newcommand{\rE}{\mathord{\text{\rm E}}}
\newcommand{\rF}{\mathord{\text{\rm F}}}

\newcommand{\Out}{\mathord{\text{\rm Out}}}
\newcommand{\Inn}{\mathord{\text{\rm Inn}}}

\newcommand{\tr}{\mathord{\text{\rm tr}}}
\newcommand{\core}{\mathord{\text{\rm c}}}

\newcommand{\Tr}{\mathord{\text{\rm Tr}}}
\newcommand{\Ball}{\mathord{\text{\rm Ball}}}
\newcommand{\spn}{\mathord{\text{\rm span}}}
\newcommand{\supp}{\mathord{\text{\rm supp}}}

\newcommand{\ovt}{\mathbin{\overline{\otimes}}}

\newcommand{\FC}{\mathord{\text{\rm FC}}}

\newcommand{\dpr}{^{\prime\prime}}

\begin{document}

\title[Structure of extensions of free Araki-Woods factors]{Structure of extensions of free Araki-Woods factors}

\begin{abstract}
We investigate the structure of crossed product von Neumann algebras arising from Bogoljubov actions of countable groups on Shlyakhtenko's free Araki--Woods factors. Among other results, we settle the questions of factoriality and Connes' type classification. We moreover provide general criteria regarding fullness and strong solidity. As an application of our main results, we obtain examples of type ${\rm III_0}$ factors that are prime, have no Cartan subalgebra and possess a maximal amenable abelian subalgebra. We also obtain a new class of strongly solid type ${\rm III}$ factors with prescribed Connes' invariants that are not isomorphic to any free Araki--Woods factors.
\end{abstract}

\author{Cyril Houdayer}

\email{cyril.houdayer@math.u-psud.fr}

\author{Benjamin Trom}

\email{benjamin.trom@math.u-psud.fr}

\address{Laboratoire de Math\'ematiques d'Orsay\\ Universit\'e Paris-Sud\\ CNRS\\ Universit\'e Paris-Saclay\\ 91405 Orsay\\ FRANCE}

\thanks{Research supported by ERC Starting Grant GAN 637601}


\maketitle

\section{Introduction and statement of the main results}

Free Araki--Woods factors were introduced by Shlyakhtenko in \cite{Sh96} using Voiculescu's free Gaussian functor \cite{Vo85, VDN92}. To any strongly continuous orthogonal representation $U : \R \curvearrowright H_\R$, one can associate a von Neumann algebra $\Gamma(H_\R, U)\dpr$, called the {\em free Araki-Woods} von Neumann algebra, that is endowed with a canonical faithful normal state $\varphi_U$, called the {\em free quasi-free} state. We refer the reader to Section \ref{section-preliminaries} for a detailed construction. Using Voiculescu's free probability theory, Shlyakhtenko settled the questions of factoriality, type classification, fullness and Connes' type ${\rm III}$ invariants for free Araki--Woods von Neumann algebras \cite{Sh96, Sh97a, Sh97b, Sh02} (see also \cite{Va04}). When $U = 1_{H_\R}$, we have $\Gamma(H_\R, 1_{H_\R})\dpr \cong \rL(\F_{\dim(H_\R)})$ and so $M = \Gamma(H_\R, 1_{H_\R})\dpr$ is a free group factor. When $U \neq 1_{H_\R}$, $\Gamma(H_\R, U)\dpr$ is a full factor of type ${\rm III}$. For that reason, free Araki--Woods factors are often regarded as type ${\rm III}$ analogues of free group factors.

To any countable group $G$ and any orthogonal representation $\pi : G \curvearrowright H_\R$ such that $U$ and $\pi$ commute (abbreviated $[U, \pi] = 0$ hereafter), one can associate the corresponding {\em free Bogoljubov} action $\sigma^\pi : G \curvearrowright \Gamma(H_\R, U)\dpr$ that preserves the free quasi-free state $\varphi_U$. We simply denote the crossed product von Neumann algebra $\Gamma(H_\R, U)\dpr \rtimes G$ by $\Gamma(U, \pi)\dpr$. We refer to the von Neumann algebra $\Gamma(U, \pi)\dpr$ as the {\em extension} of the free Araki--Woods von Neumann algebra $\Gamma(H_\R, U)\dpr$ by the countable group $G$ via the free Bogoljubov action $\sigma^\pi$.

In this paper, we investigate the structure of entensions of free Araki--Woods factors $\Gamma(U, \pi)\dpr$. Among other results, we settle the questions of factoriality and Connes' type classification. We moreover provide general criteria regarding fullness and strong solidity.  Our results generalize and strengthen some of the results obtained by the first named author \cite{Ho12b} regarding the structure of crossed product type ${\rm II_1}$ factors arising from free Bogoljubov actions of countable groups on free group factors. Moreover, we apply our results to obtain new classes of type ${\rm III}$ factors with various structural properties such as the existence of maximal amenable abelian subalgebras or the property of strong solidity, to name a few. All locally compact groups are assumed to be second countable and all (real) Hilbert spaces are assumed to be separable, unless stated otherwise.

\subsection*{Factoriality and Connes' type classification}

Our first result settles the questions of factoriality and Connes' type classification of extensions of free Araki--Woods factors $\Gamma(U, \pi)\dpr$. Let $G$ be any countable group. For every $g \in G$, we denote by $C(g) = \{hgh^{-1} \mid h \in G\}$ the conjugacy class of $g \in G$. Recall that the FC-{\em radical} of $G$ is defined by $\FC(G) = \{g \in G \mid |C(g)| < +\infty\}$. Observe that $\mathcal Z(G) < \FC(G) < G$, where $\mathcal Z(G)$ denotes the center of $G$.

\begin{letterthm}\label{letterthm-type}
Let $U : \R \curvearrowright H_\R$ be any strongly continuous orthogonal representation with $\dim H_\R \geq 2$. Let $G$ be any countable group and $\pi : G \curvearrowright H_\R$ any orthogonal representation such that $[U, \pi] = 0$. Put $M = \Gamma(U, \pi)\dpr$. The following assertions hold.
\begin{itemize}
\item [$(\rm i)$] $M$ is a factor if and only if $\pi_g \neq 1$ for every $g \in \FC(G) \setminus \{e\}$.

\item [$(\rm ii)$] Assume that $M$ is a factor. Then
$$\rT(M) = \left\{ t \in \R \mid \exists g \in \mathcal Z(G) \; \text{ such that } \; U_t = \pi_g\right \}.$$

\item [$(\rm iii)$] When $M$ is a factor, Connes' invariant $\rT(M)$ completely determines the type of $M$.
\begin{align*}
M \text{ is of type } {\rm III_1} & \quad  \Leftrightarrow \quad \rT(M) = \{0\} \\
M \text{ is of type } {\rm III_\lambda} & \quad  \Leftrightarrow \quad \rT(M) = \frac{2 \pi}{\log \lambda} \Z \quad \text{with} \quad  0 < \lambda < 1 \\
M \text{ is of type } {\rm III_0} & \quad  \Leftrightarrow \quad \rT(M) \text{ is dense in } \R \text{ and } \rT(M) \neq \R \\
M \text{ is of type } {\rm II_1} & \quad  \Leftrightarrow \quad \rT(M) = \R
\end{align*}
\item [$(\rm iv)$] When $M$ is a type ${\rm III_1}$ factor, $M$ has trivial bicentralizer.
\end{itemize}
\end{letterthm}

One of the key elements of the proof of Theorem \ref{letterthm-type} is the fact that whenever $\pi \in \mathcal O(H_\R)$ is a nontrivial orthogonal transformation that commutes with $U$, the corresponding Bogoljubov automorphism $\sigma^\pi \in \Aut( \Gamma(H_\R, U)\dpr)$ is not inner (see Lemma \ref{lem-outer}).

\subsection*{Fullness and Connes' $\boldsymbol{\tau}$ invariant}

Whenever $G$ is a locally compact group and $\rho : G \curvearrowright H_\R$ is a strongly continuous orthogonal representation, we define $\tau(\rho)$ as the weakest topology on $G$ that makes $\rho$ continuous. When $G$ is countable, we simply denote by $\tau_G$ the discrete topology on $G$. Following \cite{Co74}, we say that a factor with separable predual $M$ is {\em full} if the subgroup of inner automorphisms $\Inn(M)$ is closed in the group of all $\ast$-automorphisms $\Aut(M)$. If $M$ is full,  Connes' $\tau$ invariant $\tau(M)$ is defined as the weakest topology on $\R$ that makes the modular homomorphism $\delta_M : \R \to \Out(M)$ continuous. 

Our second result shows that the extension $\Gamma(U, \pi)\dpr$ is a full factor whenever $\pi$ is faithful and $\pi(G)$ is discrete in $\mathcal O(H_\R)$ with respect to the strong topology. This result extends \cite[Theorem A]{Ho12b} to the type ${\rm III}$ setting. Assuming moreover that $G$ is infinite and that the weakest topology on $\R \times G$ that makes the representation $\rho : \R \times G \curvearrowright H_\R$ continuous is $\tau(U) \times \tau_G$, we can compute Connes' invariant $\tau(M)$. This phenomenon is unique to the type ${\rm III}$ setting and has no analogue in the realm of type ${\rm II_1}$ factors.

\begin{letterthm}\label{letterthm-full}
Let $U : \R \curvearrowright H_\R$ be any strongly continuous orthogonal representation with $\dim H_\R \geq 2$. Let $G$ be any countable group and $\pi : G \curvearrowright H_\R$ any faithful orthogonal representation such that $[U, \pi] = 0$.  Define the strongly continuous orthogonal representation $\rho : \R \times G \curvearrowright H_\R$ by $\rho_{(t, g)} = U_t \pi_g$ for every $t \in \R$ and every $g \in G$. Put $M = \Gamma(U, \pi)\dpr$. 
\begin{itemize}
\item [$(\rm i)$] Assume that $\tau(\pi) = \tau_G$. Then $M$ is a full factor. 

\item [$(\rm ii)$] Assume that $G$ is infinite and that $\tau(\rho) =  \tau(U) \times \tau_G$. Then $M$ is a full factor and $\tau(M) = \tau(U)$.
\end{itemize}
\end{letterthm}

The proof of Theorem \ref{letterthm-full} uses a combination of Popa's asymptotic orthogonality property \cite{Po83}, $\varepsilon$-orthogonality techniques \cite{Ho12a, Ho12b} and modular theory of ultraproduct von Neumann algebras \cite{AH12}.

We should point out that Theorem \ref{letterthm-full} does not rely on Marrakchi's result \cite[Theorem B]{Ma16} (see also \cite{Jo81} for the tracial case). Recall that for any full factor $N$, any countable group  $G$ and any outer action $\sigma : G \curvearrowright N$ such that the image of $\sigma(G)$ is discrete in $\Out(N)$, the crossed product $M = N \rtimes G$ is a full factor by \cite[Theorem B]{Ma16}. The condition that the image of $\sigma(G)$ is discrete in $\Out(N)$ is rather difficult to check in general as it requires to understand the quotient group $\Out(N)$. For the class of Bogoljubov actions $\sigma^\pi : G \curvearrowright N$, where $N = \Gamma(H_\R, U)\dpr$, our Theorem \ref{letterthm-full} shows that the crossed product $N \rtimes G$ is a full factor under the weaker assumption that $\pi(G)$ is discrete in $\mathcal O(H_\R)$ with respect to the strong topology, or equivalently, that $\sigma^\pi(G)$ is discrete in $\Aut(N)$ with respect to the $u$-topology. 

When the countable group $G$ is {\em amenable}, combining our Theorem \ref{letterthm-full}, \cite[Theorem 3.6]{HMV16} and Marrakchi's very recent result \cite[Theorem A]{Ma18}, we obtain the following characterization.

\begin{corollary}
Let $U : \R \curvearrowright H_\R$ be any strongly continuous orthogonal representation with $\dim H_\R \geq 2$. Let $G$ be any amenable countable group and $\pi : G \curvearrowright H_\R$ any faithful orthogonal representation such that $[U, \pi] = 0$. Put $N = \Gamma(H_\R, U)\dpr \subset \Gamma(U, \pi)\dpr = M$. 

The following assertions are equivalent.
\begin{itemize}
\item [$(\rm i)$] $\tau(\pi) = \tau_G$.
\item [$(\rm ii)$] The image of $\sigma^\pi(G)$ is discrete in $\Out(N)$.
\item [$(\rm iii)$] $M$ is a full factor.
\item [$(\rm iv)$] For every directed set $I$ and every cofinal ultrafilter $\omega$ on $I$, we have $N' \cap M^\omega = \C 1$.
\end{itemize}
\end{corollary}

\subsection*{Amenable and Gamma absorption}
 
Next, we investigate absorption properties of the inclusion $\rL(G) \subset \Gamma(U, \pi)\dpr$ with respect to amenable and/or Gamma subalgebras. Recall that a $\sigma$-finite von Neumann algebra $N$ is said to have {\em property Gamma} if the central sequence algebra $N' \cap N^\omega$ is diffuse for some (or any) nonprincipal ultrafilter $\omega \in \beta(\N) \setminus \N$. Recall also that a von Neumann subalgebra $P \subset M$ is said to be {\em with expectation} if there exists a faithful normal conditional expectation $\rE_P : M \to P$. Our next result extends and strengthens \cite[Theorems D and E]{Ho12b} to the type ${\rm III}$ setting.

\begin{letterthm}\label{letterthm-maximal}
Let $U : \R \curvearrowright H_\R$ be any strongly continuous orthogonal representation. Let $G$ be any  countable group and $\pi : G \curvearrowright H_\R$ any orthogonal representation such that $[U, \pi] = 0$. Put $M = \Gamma(U, \pi)\dpr$.
\begin{itemize}
\item [$(\rm i)$] Assume that $\pi : G \curvearrowright H_\R$ is weakly mixing. Let $\rL(G) \subset P \subset M$ be any intermediate von Neumann subalgebra with expectation such that $P$ is amenable relative to $\rL(G)$ inside $M$. Then $P = \rL(G)$.

\item [$(\rm ii)$] Assume that $\pi : G \curvearrowright H_\R$ is mixing. Let $P \subset M$ be any von Neumann subalgebra with expectation and with property Gamma such that $P \cap \rL(G)$ is diffuse. Then $P \subset \rL(G)$.
\end{itemize}
\end{letterthm}

The proof of Theorem \ref{letterthm-maximal} relies on \cite[Theorem 5.1]{KV16} and \cite[Theorem 3.1]{HU15b} as well as mixing techniques for inclusions of von Neumann algebras (see Appendix \ref{appendix-mixing}). Note that  item $(\rm ii)$ of Theorem \ref{letterthm-maximal} can also be regarded as a strengthening of item $(\rm i)$ of Theorem \ref{letterthm-full} in the case when $\pi$ is {\em mixing}.

Using Theorem \ref{letterthm-maximal} as well as Theorems \ref{thm-prime} and \ref{thm-cartan}, we obtain examples of type ${\rm III_0}$ factors, with prescribed Connes' $\rT$ invariant, that are prime, have no Cartan subalgebra and possess a maximal amenable abelian subalgebra.

\begin{letterapp}\label{application-1}
Let $U : \R \curvearrowright H_\R$ be any mixing strongly continuous orthogonal representation. Let $G \subset \R$ be any countable dense subgroup and put $\pi = U|_G$. 

Then $M = \Gamma(U, \pi)\dpr$ is a type ${\rm III_0}$ factor such that $\rT(M) = G$. Moreover, $M$ is prime, $M$ has no Cartan subalgebra and $\rL(G) \subset M$ is maximal amenable.
\end{letterapp}

We would like to point out that all previously known examples of maximal amenable abelian subalgebras with expectation $A \subset M$  in type ${\rm III}$ factors (see \cite{Ho14, HU15a, BH16}) require the intermediate amenable subalgebra $A \subset P \subset M$ to be also with expectation. That is why the terminology ``maximal amenable with expectation" was used in \cite{Ho14, HU15a, BH16}.  Application \ref{application-1} provides the first concrete class of abelian subalgebras with expectation in type ${\rm III}$ factors that are {\em genuinely} maximal amenable, that is, with no further assumption on the intermediate amenable subalgebra $A \subset P \subset M$. Indeed, for the inclusions in Application \ref{application-1}, any intermediate von Neumann subalgebra $\rL(G) \subset P \subset M$ is automatically with expectation and so we may apply Theorem \ref{letterthm-maximal}.

\subsection*{Strong solidity}

Following \cite{OP07, BHV15}, a $\sigma$-finite von Neumann algebra $N$ is said to be {\em strongly solid} if for any diffuse amenable subalgebra with expectation $Q \subset N$, the normalizer $\mathcal N_N(Q)\dpr \subset N$ of $Q$ inside $N$ stays amenable, where $\mathcal N_N(Q) = \left\{ u \in \mathcal U(N) \mid uQu^* = Q\right\}$. In their breakthrough article \cite{OP07}, Ozawa--Popa famously proved that free group factors are strongly solid. These were the first class of strongly solid type ${\rm II_1}$ factors in the literature. Recently, generalizing the methods of Ozawa--Popa, Boutonnet--Houdayer--Vaes  \cite{BHV15} showed that free Araki--Woods factors are strongly solid, thus obtaining the first class of strongly solid type ${\rm III}$ factors.

Our next result shows that when $G$ is amenable and $\pi$ is faithful and mixing, the extension $\Gamma(U, \pi)\dpr$ is strongly solid. We refer the reader to \cite{Ca18, Is18} for other examples of strongly solid type ${\rm III}$ factors.

\begin{letterthm}\label{letterthm-strongly-solid}
Let $U : \R \curvearrowright H_\R$ be any strongly continuous orthogonal representation. Let $G$ be any amenable countable group and $\pi : G \curvearrowright H_\R$ any faithful mixing orthogonal representation such that $[U, \pi] = 0$.

Then $\Gamma(U, \pi)\dpr$ is a strongly solid factor. 
\end{letterthm}

The proof of Theorem \ref{letterthm-strongly-solid} uses Popa's deformation/rigidity theory. More precisely, it relies on \cite[Theorem 3.7]{BHV15} and \cite[Theorem A]{Is18} as well as mixing techniques for inclusions of von Neumann algebras (see Appendix \ref{appendix-mixing}).

We use Theorem \ref{letterthm-strongly-solid} to obtain a new class of strongly solid type ${\rm III}$ factors, with prescribed Connes' invariants, that are not isomorphic to any free Araki--Woods factor. 

\begin{letterapp}\label{application-2}
Let $U : \R \curvearrowright H_\R$ be any strongly continuous orthogonal representation such that $U \neq 1_{H_\R}$ and such that $U$ has a nonzero invariant vector. Let $\pi : \Z \curvearrowright K_\R$ be any mixing orthogonal representation such that the spectral measure of $\bigoplus_{n \geq 1} \pi^{\otimes n}$ is singular with respect to the Haar measure on $\mathbf T$.

Put $M = \Gamma(U \otimes 1_{K_\R}, 1_{H_\R} \otimes \pi)\dpr$. Then $M$ is a strongly solid type ${\rm III}$ factor that has the complete metric approximation property and the Haagerup property and such that $\rT(M) = \ker(U)$ and $\tau(M) = \tau(U)$. Moreover, $M$ is not isomorphic to any free Araki--Woods factor.
\end{letterapp}

\subsection*{Acknowledgments} Cyril Houdayer is grateful to Yusuke Isono for thought-provoking discussions regarding \cite[Theorem A]{Is18} used in the proof of Theorem \ref{letterthm-strongly-solid}.

{\hypersetup{linkcolor=black}
\tableofcontents }

\section{Preliminaries}\label{section-preliminaries}

\subsection{Background on $\sigma$-finite von Neumann algebra}

For any von Neumann algebra $M$, we denote by $\mathcal Z(M)$ its centre, by $\mathcal U(M)$ its group of unitaries, by $\Ball(M)$ its unit ball with respect to the uniform norm $\|\cdot\|_\infty$ and by $(M, \rL^2(M), J, \rL^2(M)_+)$ its standard form \cite{Ha73}.  

Let $M$ be any $\sigma$-finite von Neumann algebra and $\varphi \in M_\ast$ any faithful state. We  write $\|x\|_\varphi = \varphi(x^* x)^{1/2}$ for every $x \in M$. Recall that on $\Ball(M)$, the topology given by $\|\cdot\|_\varphi$ 
coincides with the strong topology. We denote by $\xi_\varphi = \varphi^{1/2} \in \rL^2(M)_+$ the unique element such that $\varphi = \langle \, \cdot \, \xi_{\varphi}, \xi_{\varphi}\rangle$. Conversely, for every $\xi \in \rL^{2}(M)_{+}$, we denote by $\varphi_{\xi} = \langle \, \cdot \, \xi, \xi\rangle \in M_{\ast}$ the corresponding positive form. The mapping $M \to \rL^2(M) : x \mapsto x \xi_\varphi$ defines an embedding with dense image such that $\|x\|_\varphi = \|x \xi_\varphi\|$ for all $x \in M$.

Let $M$ be any $\sigma$-finite von Neumann algebra and $\varphi \in M_\ast$ any faithful state. We denote by $\sigma^\varphi$ the modular automorphism group of the state $\varphi$.  The centralizer $M_\varphi$ of the state $\varphi$ is by definition the fixed point algebra of $(M, \sigma^\varphi)$.  The {\em continuous core} of $M$ with respect to $\varphi$, denoted by $\core_\varphi(M)$, is the crossed product von Neumann algebra $M \rtimes_{\sigma^\varphi} \R$.  The natural inclusion $\pi_\varphi: M \to \core_\varphi(M)$ and the unitary representation $\lambda_\varphi: \R \to \core_\varphi(M)$ satisfy the {\em covariance} relation
$$
\forall x \in M, \forall t \in \R, \quad  \lambda_\varphi(t) \pi_\varphi(x) \lambda_\varphi(t)^*
  =
  \pi_\varphi(\sigma^\varphi_t(x)).
$$
Put $\rL_\varphi (\R) = \lambda_\varphi(\R)\dpr$. There is a unique faithful normal conditional expectation $\rE_{\rL_\varphi (\R)}: \core_{\varphi}(M) \to \rL_\varphi(\R)$ satisfying $\rE_{\rL_\varphi (\R)}(\pi_\varphi(x) \lambda_\varphi(t)) = \varphi(x) \lambda_\varphi(t)$. The faithful normal semifinite weight $\rL^\infty(\R)^+ \to [0, +\infty] : f \mapsto \int_{\R} \exp(-s)f(s) \, {\rm d}s$ gives rise to a faithful normal semifinite weight $\Tr_\varphi$ on $\rL_\varphi(\R)$ via the Fourier transform. The formula $\Tr_\varphi = \Tr_\varphi \circ \rE_{\rL_\varphi (\R)}$ extends it to a faithful normal semifinite trace on $\core_\varphi(M)$.

Because of Connes' Radon--Nikodym cocycle theorem \cite[Th\'eor\`eme 1.2.1]{Co72} (see also \cite[Theorem VIII.3.3]{Ta03}), the semifinite von Neumann algebra $\core_\varphi(M)$ together with its trace $\Tr_\varphi$ does not depend on the choice of $\varphi$ in the following precise sense. If $\psi$ is another faithful normal state on $M$, there is a canonical surjective $\ast$-isomorphism
$\Pi_{\varphi,\psi} : \core_\psi(M) \to \core_{\varphi}(M)$ such that $\Pi_{\varphi,\psi} \circ \pi_\psi = \pi_\varphi$ and $\Tr_\varphi \circ \Pi_{\varphi,\psi} = \Tr_\psi$. Note however that $\Pi_{\varphi,\psi}$ does not map the subalgebra $\rL_\psi(\R) \subset \core_\psi(M)$ onto the subalgebra $\rL_\varphi(\R) \subset \core_\varphi(M)$ (and thus we use the symbol $\rL_\varphi(\R)$ instead of the usual $\rL(\R)$).

\subsection{Ultraproduct von Neumann algebras}

Let $M$ be any $\sigma$-finite von Neumann algebra. Let $J$ be any nonempty directed set and ${\omega}$ any {\em cofinal} ultrafilter on $J$, that is, for all $j_0 \in J$, we have $\{j \in J : j \geq j_0\} \in \omega$. Define
\begin{align*}
\mathfrak I_{\omega}(M) &= \left\{ (x_j)_j \in \ell^\infty(J, M) \mid \lim_{j \to \omega} \|x_j \zeta\| = \lim_{j \to \omega} \|\zeta x_{j}\| =  0, \forall \zeta \in \rL^{2}(M) \right\} \\
\mathfrak M^{\omega}(M) &= \left\{ (x_j)_j \in \ell^\infty(J, M) \mid  (x_j)_j \, \mathfrak I_{\omega}(M) \subset \mathfrak I_{\omega}(M) \text{ and } \mathfrak I_{\omega}(M) \, (x_j)_j \subset \mathfrak I_{\omega}(M)\right\}.
\end{align*}
Observe that $\mathfrak I_{\omega}(M) \subset \mathfrak M^{\omega}(M)$. The {\em multiplier algebra} $\mathfrak M^{\omega}(M)$ is a $\rC^*$-algebra and $\mathfrak I_{\omega}(M) \subset \mathfrak M^{\omega}(M)$ is a norm closed two-sided ideal. Following \cite[\S 5.1]{Oc85}, we define the {\em ultraproduct von Neumann algebra} by $M^{\omega} = \mathfrak M^{\omega}(M) / \mathfrak I_{\omega}(M)$, which is indeed known to be a von Neumann algebra. Observe that the proof given in \cite[5.1]{Oc85} for the case when $J = \N$ and $\omega \in \beta(\N) \setminus \N$ applies {\em mutatis mutandis}. We denote the image of $(x_j)_j \in \mathfrak M^{\omega}(M)$ by $(x_j)^{\omega} \in M^{\omega}$.

\subsection{Extensions of free Araki--Woods factors}

Let $U : \R \curvearrowright H_\R$ be any strongly continuous orthogonal representation. Denote by $H = H_{\R} \otimes_{\R} \C = H_\R \oplus {\rm i} H_\R$ the complexified Hilbert space, by $I : H \to H : \xi + {\rm i} \eta \mapsto \xi - {\rm i} \eta$ the canonical involution on $H$ and by $A$ the infinitesimal generator of $U : \R \curvearrowright H$, that is, $U_t = A^{{\rm i}t}$ for all $t \in \R$. We have $IAI = A^{-1}$. Observe that $j : H_{\R} \to H :\zeta \mapsto (\frac{2}{A^{-1} + 1})^{1/2}\zeta$ defines an isometric embedding of $H_{\R}$ into $H$. Put $K_{\R} = j(H_{\R})$. It is easy to see that $K_\R \cap {\rm i} K_\R = \{0\}$ and that $K_\R + {\rm i} K_\R$ is dense in $H$. Write $T = I A^{-1/2}$. Then $T$ is a conjugate-linear closed invertible operator on $H$ satisfying $T = T^{-1}$ and $T^*T = A^{-1}$. Such an operator is called an {\it involution} on $H$. Moreover, we have $\dom(T) = \dom(A^{-1/2})$ and $K_\R = \{ \xi \in \dom(T) \mid T \xi = \xi \}$. In what follows, we simply write
$$\forall \xi, \eta \in K_\R, \quad \overline{\xi + {\rm i} \eta} = T(\xi + {\rm i} \eta) = \xi - {\rm i} \eta.$$
We introduce the \emph{full Fock} space of $H$ by
\begin{equation*}
\mathcal{F}(H) =\C\Omega \oplus \bigoplus_{n = 1}^{\infty} H^{\otimes n}.
\end{equation*}
The unit vector $\Omega$ is called the \emph{vacuum} vector. For all $\xi \in H$, define the {\it left creation} operator $\ell(\xi) : \mathcal{F}(H) \to \mathcal{F}(H)$ by
\begin{equation*}
\left\{ 
{\begin{array}{l} \ell(\xi)\Omega = \xi \\ 
\ell(\xi)(\xi_1 \otimes \cdots \otimes \xi_n) = \xi \otimes \xi_1 \otimes \cdots \otimes \xi_n.
\end{array}} \right.
\end{equation*}
We have $\|\ell(\xi)\|_\infty = \|\xi\|$ and $\ell(\xi)$ is an isometry if $\|\xi\| = 1$. For all $\xi \in K_\R$, put $W(\xi) = \ell(\xi) + \ell(\xi)^*$. The crucial result of Voiculescu \cite[Lemma 2.6.3]{VDN92} is that the distribution of the self-adjoint operator $W(\xi)$ with respect to the vector state $\varphi_U = \langle \, \cdot \,\Omega, \Omega\rangle$ is the semicircular law of Wigner supported on the interval $[-\|\xi\|, \|\xi\|]$. 

Following \cite{Sh96}, the \emph{free Araki--Woods} von Neumann algebra associated with $U : \R \curvearrowright H_\R$ is defined by
\begin{equation*}
\Gamma(H_\R, U)\dpr = \left\{W(\xi) \mid \xi \in K_{\R} \right\}\dpr.
\end{equation*}
The vector state $\varphi_U = \langle \, \cdot \,\Omega, \Omega\rangle$ is called the {\it free quasi-free} state and is faithful on $\Gamma(H_\R, U)\dpr$. Let $\xi, \eta \in K_\R$ and write $\zeta = \xi + {\rm i} \eta$. Put
\begin{equation*}
W(\zeta) = W(\xi) +  {\rm i} W(\eta) = \ell(\zeta) + \ell(\overline \zeta)^*.
\end{equation*}

It is easy to see that for all $n \geq 1$ and all $\zeta_1, \dots, \zeta_n \in K_\R + {\rm i} K_\R$, $\zeta_1 \otimes \cdots \otimes \zeta_n \in \Gamma(H_\R, U)\dpr \Omega$. When $\zeta_1, \dots, \zeta_n$ are all nonzero, we will denote by $W(\zeta_1 \otimes \cdots \otimes  \zeta_n) \in \Gamma(H_\R, U)\dpr$ the unique element such that 
\[\zeta_1 \otimes \cdots \otimes \zeta_n = W(\zeta_1 \otimes \cdots \otimes \zeta_n) \Omega.\]
Such an element is called a {\it reduced word}. By \cite[Proposition 2.1 (i)]{HR14} (see also \cite[Proposition 2.4]{Ho12a}), the reduced word $W(\zeta_1 \otimes \cdots \otimes \zeta_n)$ satisfies the {\em Wick formula} given by
$$W(\zeta_1 \otimes \cdots \otimes \zeta_n) = \sum_{k = 0}^n \ell(\zeta_1) \cdots \ell(\zeta_k) \ell(\overline \zeta_{k + 1})^* \cdots \ell(\overline \zeta_n)^*.$$

Since inner products are assumed to be linear in the first variable, we have $\ell(\xi)^*\ell(\eta) = \overline{\langle \xi, \eta\rangle} 1 = \langle \eta, \xi \rangle 1$ for all $\xi, \eta \in H$. In particular, the Wick formula from \cite[Proposition 2.1 (ii)]{HR14} is 
\begin{align*}
& W(\xi_1 \otimes \cdots \otimes \xi_r) W(\eta_1 \otimes \cdots \otimes \eta_s) \\
&= W(\xi_1 \otimes \cdots \otimes \xi_r \otimes \eta_1 \otimes \cdots \otimes \eta_s) + 
      \overline{\langle \overline \xi_r, \eta_1\rangle} \, W(\xi_1 \otimes \cdots \otimes \xi_{r - 1}) W(\eta_2 \otimes \cdots \otimes \eta_s)
\end{align*}
for all $\xi_1, \dots, \xi_r, \eta_1, \dots, \eta_s \in K_\R + {\rm i} K_\R$. We will repeatedly use this fact throughout. We refer to \cite[Section 2]{HR14} for further details.

The modular automorphism  group $\sigma^{\varphi_U}$ of the free quasi-free state $\varphi_U$ is given by $\sigma^{\varphi_U}_{t} = \Ad(\mathcal{F}(U_t))$, where $\mathcal{F}(U_t) = 1_{\C \Omega} \oplus \bigoplus_{n \geq 1} U_t^{\otimes n}$. In particular, it satisfies
\begin{equation*}
\forall n \in \N, \forall \zeta_1, \dots, \zeta_n \in K_\R +{\rm i} K_\R, \forall t \in \R, \quad \sigma_{t}^{\varphi_U}(W(\zeta_1 \otimes \cdots \otimes \zeta_n))  =  W(U_t \zeta_1 \otimes \cdots \otimes U_t \zeta_n). 
\end{equation*}


Let now $G$ be any countable group and $\pi:G\curvearrowright H_\R$ any orthogonal representation such that $U$ and $\pi$ commute (hereafter abbreviated $[U, \pi] = 0$). Denote by $\pi:G\curvearrowright H$  the corresponding unitary representation. Using the Wick formula, for all $n \geq 0$, all $\zeta_1, \dots, \zeta_n \in K_\R + {\rm i} K_\R$ and all $g \in G$, we have 
$$\Ad(\mathcal F(\pi_g))(W(\zeta_1 \otimes \cdots \otimes \zeta_n)) = W(\pi_g(\zeta_1) \otimes \cdots \otimes \pi_g(\zeta_n)).$$
The action $\Ad(\mathcal F(\pi)) : G \curvearrowright \mathbf B(\mathcal F(H))$ leaves the free Araki--Woods von Neumann algebra $\Gamma(H_\R, U)\dpr$ globally invariant. We use the following terminology.
\begin{df}
The action $\sigma^\pi : G \curvearrowright \Gamma(H_\R, U)\dpr$ defined by $\sigma^\pi_g = \Ad(\mathcal F(\pi_g))$ for every $g \in G$ is called the {\em free Bogoljubov} action associated with the orthogonal representation $\pi : G \curvearrowright H_\R$. The action $\sigma^\pi$ preserves the quasi-free state $\varphi_U$ and commute with its modular automorphism group $\sigma^{\varphi_U}$, that is, 
$$\forall g \in G, \forall t \in \R, \quad \varphi_U = \varphi_U \circ \sigma^\pi_g \quad \text{and} \quad \sigma^{\varphi_U}_t \circ \sigma^\pi_g = \sigma^\pi_g \circ \sigma^{\varphi_U}_t.$$
\end{df}
We simply denote by $\Gamma(U, \pi)\dpr = \Gamma(H_\R, U)\dpr \rtimes_{\sigma^\pi} G$ the corresponding crossed product von Neumann algebra. We refer to $\Gamma(U, \pi)\dpr$ as the {\em extension} of the free Araki--Woods von Neumann algebra $\Gamma(H_\R, U)\dpr$ by the countable group $G$ via the free Bogoljubov action $\sigma^\pi : G \curvearrowright \Gamma(H_\R, U)\dpr$. Put $N =  \Gamma(H_\R, U)\dpr$ and $M = \Gamma(U, \pi)\dpr$ so that $M = N \rtimes G$. Denote by $\rE_N : M \to N$ the canonical faithful normal conditional expectation and by $\varphi = \varphi_U \circ \rE_N$ the canonical faithful normal state on $M$. We identify the standard form $\rL^2(M)$ with $\mathcal F(H) \otimes \ell^2(G)$ via the unitary mapping 
$$U : \rL^2(M) \to \mathcal F(H) \otimes \ell^2(G) : W(\zeta_1 \otimes \cdots \otimes \zeta_n) u_g \xi_\varphi \mapsto \zeta_1 \otimes \cdots \otimes \zeta_n \otimes \delta_g$$
where $n \geq 0$, $\zeta_1, \dots, \zeta_n \in K_\R +{\rm i} K_\R$, $g \in G$.

\subsection{Intertwining theory}

Let $M$ be any $\sigma$-finite von Neumann algebra and $A \subset 1_A M 1_A$, $B \subset 1_B M 1_B$ any von Neumann subalgebras with expectation. Following \cite{Po01, Po03, HI15}, we say that $A$ {\em embeds with expectation into} $B$ {\em inside} $M$ and write $A \preceq_M B$, if there exist projections $e \in A$ and $f \in B$, a nonzero partial isometry $v \in eMf$ and a unital normal $\ast$-homomorphism $\theta: eAe \to fBf$ such that the inclusion $\theta(eAe) \subset fBf$ is with expectation and $av = v\theta(a)$ for all $a \in eAe$.

We will need the following technical result that is essentially contained in \cite[Lemma 2.6]{HV12}.

\begin{lem}\label{lem-intertwining}
Let $M$ be any von Neumann algebra with separable predual. Let $A \subset 1_A M 1_A$, $B \subset M$ be any von Neumann  subalgebras with expectation. Assume that $B$ is of type ${\rm I}$. If $A \npreceq_M B$, there exists a diffuse abelian subalgebra with expectation $D \subset A$ such that $D \npreceq_M B$.
\end{lem}

\begin{proof}
The proof of \cite[Lemma 2.6]{HV12} applies {\em mutatis mutandis} to the case when $B$ is a type ${\rm I}$ von Neumann algebra.
\end{proof}

\subsection{Relative amenability}

Let $M$ be any $\sigma$-finite von Neumann algebra and $A \subset 1_A M 1_A$, $B \subset M$ any von Neumann subalgebras with expectation. Following \cite{OP07, HI17}, we say that $A$ is {\em amenable relative to $B$ inside $M$} and write $A \lessdot_M B$ if there exists a conditional expectation $\Phi : 1_A \langle M, B \rangle 1_A \to A$ such that the restriction $\Phi |_{1_A M 1_A} : 1_A M 1_A \to A$ is faithful and normal. 

Fix now $\rE_A : 1_A M 1_A \to A$ a faithful normal conditional expectation. Using \cite[Remark 3.3]{Is17}, $A \lessdot_M B$ if and only if there exists a conditional expectation $\Phi : 1_A \langle M, B \rangle 1_A \to A$ such that $\Phi |_{1_A M 1_A} = \rE_A$.

\section{Factoriality and Connes' type classification}

We start by proving the following well known fact about orthogonal representations of abelian locally compact second countable (lcsc) groups.

\begin{prop}\label{prop-irreducible}
Let $G$ be any abelian lcsc group and $\rho : G \curvearrowright H_\R$ any irreducible strongly continuous orthogonal representation. Then $\dim(H_\R) \in \{1, 2\}$.

If $\dim(H_\R) = 1$, there exists a continuous homomorphism $\varepsilon : G \to \{-1, 1\}$ such that for all $g \in G$, we have $\rho_g = \varepsilon_g 1_{H_\R}$.

If $\dim(H_\R) = 2$, there exist a Borel map $\psi : G \to \R$ such that for all $g, h \in G$, we have $\psi(gh) - \psi(g) - \psi(h) \in 2\pi \Z$, and there exists an orthonormal basis of $H_\R$ such that $\rho : G \curvearrowright H_\R$ has the following form:
$$\forall g \in G, \quad \rho_g = \begin{pmatrix}
\cos(\psi(g)) & \sin(\psi(g)) \\
-\sin(\psi(g)) & \cos(\psi(g))
\end{pmatrix}.$$
\end{prop}

\begin{proof}
We denote by $\rho_{\C} : G \curvearrowright H_\C$ the complexified strongly continuous  unitary representation. Denote by $J : H_\C \to H_\C$ the canonical conjugation. Observe that $\rho_{\C}(g) J = J \rho_{\C}(g)$ for every $g \in G$. The fact that the orthogonal representation $\rho$ is irreducible translates into the following fact for $\rho_{\C}$: the only closed subspaces of $H_{\C}$ that are invariant under both $\rho_{\C}(G)$ and $J$ are $\{0\}$ and $H_\C$.

Since $G$ is an abelian lcsc group, we may consider the spectral measure $\rE_{\rho_{\C}} : \mathcal B(\widehat G) \to \mathcal P(H_\C)$ where $\mathcal B(\widehat G)$ is the $\sigma$-algebra of Borel subsets of $\widehat G$ and $\mathcal P(H_\C)$ is the lattice of projections of $\mathbf B(H_\C)$. Then we have 
$$\forall g \in G, \quad \rho(g) = \int_{\widehat G} \chi(g) \, {\rm d}\rE_{\rho_{\C}}(\chi).$$
Since $\rho_{\C}(g) = J \rho_{\C}(g) J$ for every $g \in G$, we have $\rE_{\rho_\C}(\overline B) = J\rE_{\rho_\C}(B)J$ for every $B \in \mathcal B(\widehat G)$. 

We claim that there exists $\chi \in \widehat G$ such that $\supp(\rE_{\rho_\C}) = \{\chi, \overline \chi\}$. Indeed, otherwise we can find $\chi_{1}, \chi_{2} \in \supp(\rE_{\rho_\C})$ such that $\{\chi_1, \overline{\chi_1}\} \cap \{\chi_2, \overline{\chi_2}\} = \emptyset$. We can then find open neighborhoods $O_1, O_2 \subset \widehat G$ of $\chi_1, \chi_2$ respectively such that $\left(O_1 \cup \overline{O_1} \right) \cap \left(O_2 \cup \overline{O_2} \right) = \emptyset$. Put $B = O_1 \cup \overline{O_1}$. Then $B = \overline B$ and $\rE_{\rho_{\C}}(B) \neq 0$. Since $O_2 \cup \overline{O_2} \subset \widehat G \setminus B$, we have $\rE_{\rho_{\C}}(\widehat G \setminus B) \neq 0$. It follows that the range $K$ of $\rE_{\rho_{\C}}(B)$ is a subspace of $H_\C$ that is invariant under $\rho_\C(G)$ and $J$ and such that $K \neq \{0\}$ and $K \neq H_{\C}$. This contradicts the irreducibility of the orthogonal representation $\rho$. Thus,  there exists $\chi \in \widehat G$ such that $\supp(\rE_{\rho_\C}) = \{\chi, \overline\chi\}$. 

Firstly, assume that $\overline \chi = \chi$. Since the orthogonal representation $\rho$ is irreducible, $\rE_{\rho_\C}(\{\chi\})$ is necessarily a rank one projection and so $\dim_{\C} (H_\C) = 1$. This implies that $\dim_{\R}(H_\R) = 1$. Then there exists a continuous homomorphism $\varepsilon : G \to \{-1, 1\}$ such that for all $g \in G$, we have $\rho_g = \varepsilon_g 1_{H_\R}$.

Secondly, assume that $\overline \chi \neq \chi$. Since the orthogonal representation $\rho$ is irreducible, $\rE_{\rho_\C}(\{\chi\})$ is necessarily a rank one projection and so $\dim_{\C} (H_\C) = 2$. This implies that $\dim_{\R}(H_\R) = 2$. Consider the one-to-one Borel map $f : \mathbf T \to ]- \pi, \pi]$ such that $\exp({\rm i}f(z)) = z$ for every $z \in \mathbf T$. Define the Borel map $\psi : G \to ]-\pi, \pi]$ by $\psi = f \circ \chi$. Then we have $\psi(gh) - \psi(g) - \psi(h) \in 2\pi \Z$ for all $g, h \in G$. Moreover, there exists an orthonormal basis of $H_\R$ such that $\rho : G \curvearrowright H_\R$ has the following form:
$$\forall g \in G, \quad \rho_g = \begin{pmatrix}
\cos(\psi(g)) & \sin(\psi(g)) \\
-\sin(\psi(g)) & \cos(\psi(g))
\end{pmatrix}.$$
This finishes the proof of Proposition \ref{prop-irreducible}.
\end{proof}

We generalize \cite[Theorem 5.1]{HS09} to the setting of Bogoljubov transformations of free Araki--Woods factors.

\begin{lem}\label{lem-outer}
 Let $U : \R \curvearrowright H_\R$ be any strongly continuous orthogonal representation with $\dim H_\R \geq 1$. Put $(N, \varphi) = (\Gamma(H_\R, U)\dpr, \varphi_U)$. Let $\pi \in \mathcal O(H_\R)$ be any orthogonal transformation such that $[U, \pi] = 0$. 

If $\Ad(\mathcal F(\pi)) \in \Inn(N)$, then $\pi = 1$.
\end{lem}

\begin{proof}
Assume that $\theta = \Ad(\mathcal F(\pi)) \in \Inn(N)$. Let $u \in \mathcal U(N)$ such that $\theta = \Ad(u)$. We show that $\pi = 1$. Since $[U, \pi] = 0$, we may define the strongly continuous orthogonal representation $\rho : \R \times \Z \curvearrowright H_\R$ by $\rho_{(t, n)} = U_t \pi^n$. There are two cases to consider.

First, assume that $\rho$ is {\em reducible} and write $H_\R = H_\R^1 \oplus H_\R^2$ where $H_\R^i \subset H_\R$ is a nonzero $\rho$-invariant subspace for every $i \in \{1, 2\}$. Write $U_i = U|_{H_\R^i}$ for every $i \in \{1, 2\}$. Then \cite[Theorem 2.11]{Sh96} implies that $(N, \varphi) = (\Gamma(H_\R^1, U_1)\dpr, \varphi_{U_1}) \ast (\Gamma(H_\R^2, U_2)\dpr, \varphi_{U_2})$. For every $i \in \{1, 2\}$, we moreover have
$$u\, \Gamma(H_\R^i, U_i)\dpr u^*= \theta( \Gamma(H_\R^i, U_i)\dpr) = \Gamma(\pi(H_\R^i), U_i)\dpr = \Gamma(H_\R^i, U_i)\dpr.$$
Then \cite[Proposition 3.1]{Ue10} implies that $u \in \Gamma(H_\R^1, U_1)\dpr \cap \Gamma(H_\R^2, U_2)\dpr$ and so $u \in \mathbf T 1$. Thus, $\theta = \id_N$ and so $\pi = 1$.

Secondly, assume that $\rho$ is {\em irreducible}. Since $\R \times \Z$ is an abelian lcsc group, we have $\dim(H_\R) \in \{1, 2\}$ by Proposition \ref{prop-irreducible}. If $\dim(H_\R) = 1$, we have $\theta = \id_N$ and thus $\pi = 1$. If $\dim(H_\R) = 2$, there exist $0 < \lambda \leq 1$ and $\mu > 0$ and an orthonormal basis of $H_\R$ such that $\rho : \R \times \Z \curvearrowright H_\R$ has the following form: 
$$\forall (t, n) \in \R \times \Z, \quad \rho_{(t, n)} = \begin{pmatrix}
\cos(t \log \lambda + n \log\mu) & -\sin(t \log \lambda + n \log\mu) \\
\sin(t \log \lambda + n \log\mu) & \cos(t \log \lambda + n \log\mu)
\end{pmatrix}.$$
If $\lambda = 1$, $U$ is trivial and \cite[Theorem 5.1]{HS09} implies that $\pi = 1$. If $0 < \lambda < 1$, using \cite[Section 4]{Sh96}, we know that $N$ is a type ${\rm III_\lambda}$ factor. Moreover, we have $\pi = \rho_{(0, 1)} = U_t$ where $t = \frac{\log \mu}{ \log \lambda}$. Since $\rT(N) = \frac{2 \pi}{\log \lambda} \Z$ and since $\sigma_t^\varphi = \theta = \Ad(u)$, there exists $k \in \Z$ such that $\frac{\log \mu}{\log \lambda} = t = \frac{2\pi k}{\log \lambda}$. Then $\log \mu = 2 \pi k$ and so $\pi = \rho_{(0, 1)} = 1$.
\end{proof}

\begin{proof}[Proof of Theorem \ref{letterthm-type}]
Put $N = \Gamma(H_\R, U)\dpr$ and $M = \Gamma(U, \pi)\dpr$ so that $M = N \rtimes G$. Denote by $\rE_N : M \to N$ the canonical faithful normal conditional expectation. Denote by $\varphi \in N_\ast$ the free quasi-free state on $N$ and put $\psi = \varphi \circ \rE_N \in M_\ast$. 

$(\rm i)$ We start by proving the following claim.
\begin{claim}\label{claim-1}
$\rL(G)' \cap M = \rL(G)' \cap (N \rtimes \FC(G))$.
\end{claim}
Indeed, let $x \in \rL(G)' \cap M$ and write $x = \sum_h x^h u_h$ for its Fourier decomposition where $x^h = \rE_N(xu_h^*)$ for every $h \in G$. Then we have $\sigma_g^\pi(x^{g^{-1}hg}) = x^h$ for all $g, h \in G$. Since $\sigma^\pi : G \curvearrowright N$ is $\varphi$-preserving and since $\sum_h \|x^h\|_\varphi^2 = \|x\|_\psi^2 < +\infty$, it follows that $x^h = 0$ for every $h \in G\setminus \FC(G)$. Since $N \rtimes \FC(G) \subset M$ is $\sigma^\psi$-invariant, we may denote by $\rF : M \to N \rtimes \FC(G)$ the unique $\psi$-preserving conditional expectation. Then $x = \rF(x) \in N \rtimes \FC(G)$ and the claim is proven.

Assume that $\pi_g \neq 1$ for every $g \in \FC(G) \setminus \{e\}$. Since $\mathcal Z(M) \subset \rL(G)' \cap M$, we have $\mathcal Z(M) = M' \cap (N \rtimes \FC(G))$ by Claim \ref{claim-1}. By assumption and using Lemma \ref{lem-outer}, the action $\sigma^\pi : \FC(G) \curvearrowright N$ is outer. It follows that $N' \cap (N \rtimes \FC(G)) = \C1$ and so $\mathcal Z(M) = \C 1$.

Assume that $\pi_g = 1$ for some $g \in \FC(G) \setminus \{e\}$. Then $\pi_k = 1$ for every $k \in C(g)$. This implies that $x = \sum_{k \in C(g)} u_k \in \mathcal Z(M)$ and so $\mathcal Z(M) \neq \C 1$.

$(\rm ii)$ We compute Connes' invariant $\rT(M)$. Let $t \in \R$ for which there exists $g \in \mathcal Z(G)$ such that $U_t = \pi_g$. Then $\sigma_t^\varphi = \sigma^\pi_g$. By construction of the crossed product von Neumann algebra $M = N \rtimes G$, since $\rL(G) \subset M_\psi$ and since $g \in \mathcal Z(G)$, we have $\sigma_t^{\psi} = \Ad(u_g)$ and so $t \in \rT(M)$. 

Conversely, let $t \in \rT(M)$. Then there exists $u \in \mathcal U(M)$ such that $\sigma_t^\psi = \Ad(u)$. Since $\rL(G) \subset M_\psi$, we have $u \in \rL(G)' \cap M$ and so $u \in N \rtimes \FC(G)$ by Claim \ref{claim-1}. As $\sigma_t^\psi$ leaves $N$ globally invariant, we have $u \in \mathcal N_{N \rtimes \FC(G)}(N)$. By assumption and using Lemma \ref{lem-outer}, the action $\sigma^\pi : \FC(G) \curvearrowright N$ is outer. Then there exist $v \in \mathcal U(N)$ and $g \in \FC(G)$ such that $u = v u_g$ (see e.g.\ \cite[Corollary 3.11]{BB16}). Thus, we have $\sigma_t^\varphi = \Ad(v) \circ \sigma^\pi_g$ and so $\sigma_t^\varphi \circ \sigma^\pi_{g^{-1}} \in \Inn(N)$. Since $\sigma_t^\varphi \circ \sigma^\pi_{g^{-1}} = \Ad(\mathcal F(U_t \pi_g^*))$, Lemma \ref{lem-outer} implies that $U_t = \pi_g$. Since $[U, \pi] = 0$, we have $\pi_g \in \mathcal Z(\pi(G))$. Since $M$ is a factor, $\pi|_{\FC(G)}$ is faithful by item $(\rm i)$. Altogether, this implies that $g \in \mathcal Z(G)$.

$(\rm iii)$ We now prove that Connes' invariant $\rT(M)$ completely determines the type of $M$. Denote by $\core(N)$ (resp.\ $\core(M)$) the continuous core of $N$ (resp.\ $M$). We canonically have $\core(M) = \core(N) \rtimes G$ where the action $\core(\sigma^\pi) : G \curvearrowright \core(N)$ is given by $\core(\sigma^\pi)_g(\pi_\varphi(x) \lambda_\varphi(t)) = \pi_\varphi(\sigma^\pi_g(x))\lambda_\varphi(t)$ for every $g \in G$, every $t \in \R$ and every $x \in N$. We prove the following claim.

\begin{claim}\label{claim-2}
$\rL(G)' \cap \core(M) = \rL(G)' \cap (\core(N) \rtimes \FC(G))$.
\end{claim}
Indeed, let $p \in \rL_\varphi(\R) \subset \core(N)$ be any nonzero finite trace projection. Since $p \in \rL(G)' \cap \core(M)$, it follows that the restriction $\core(\sigma^\pi) : G \curvearrowright p \, \core(N) \, p$ is a trace preserving action on a tracial von Neumann algebra and $p \,\core(M) \, p = p \,\core(N) \,p \rtimes G$. The same proof as in Claim \ref{claim-1} shows that 
\begin{align*}
p (\rL(G)' \cap \core(M) ) p &= \rL(G)' \cap p\,\core(M)\,p \\
&= \rL(G)' \cap (p\,\core(N)\,p \rtimes \FC(G)) \\
&= \rL(G)' \cap p(\core(N) \rtimes \FC(G))p \\
&= p (\rL(G)' \cap (\core(N) \rtimes \FC(G))) p.
\end{align*} 
Since we can find an increasing sequence of nonzero finite trace projections $p_k \in \rL_\varphi(\R)$ such that $p_k \to 1$ strongly, the claim is proven.

If $M$ is of type ${\rm III_1}$, then $\rT(M) = \{0\}$ by \cite[Th\'eor\`eme 3.4.1]{Co72}. Conversely, assume that $\rT(M) = \{0\}$. Recall that by \cite{Sh96}, we have $\rT(N) = \{t \in \R \mid U_t = 1\}$ and Connes' invariant $\rT(N)$ completely determines the type of $N$. Since $\rT(N) \subset \rT(M)$, we also have $\rT(N) = \{0\}$. This implies that $N$ is of type ${\rm III_1}$. By combining \cite[Proposition 5.4]{HS88}, Lemma \ref{lem-outer} and item $(\rm ii)$, we have that $\core(\sigma^\pi) : \FC(G) \curvearrowright \core(N)$ is an outer action. This implies that $\core(N)' \cap (\core(N) \rtimes \FC(G)) = \C 1$. By Claim \ref{claim-2}, we have $\mathcal Z(\core(M)) = \core(M)' \cap (\core(N) \rtimes \FC(G))$ and so $\mathcal Z(\core(M)) = \C1$. This implies that $\core(M)$ is a factor and so $M$ is of type ${\rm III_1}$.

Let $0 < \lambda < 1$ and put $T = \frac{2 \pi}{\log \lambda}$. If $M$ is of type ${\rm III_\lambda}$, then $\rT(M) = T \Z$ by \cite[Th\'eor\`eme 3.4.1]{Co72}. Conversely, assume that $\rT(M) = T \Z$. Since $\rT(N) \subset \rT(M)$, we have $\rT(N) = \kappa T \Z$ for some $\kappa \in \N$. Since $M$ is a factor, $\pi |_{\FC(G)}$ is faithful and we may denote by $g \in \mathcal Z(G)$ the unique element such that $U_T = \pi_g$. Moreover, the map $\varepsilon : \rT(M) \to \mathcal Z(G) : kT \mapsto g^k$ is a well-defined group homomorphism. 

Firstly, assume that $\kappa = 0$, so that $N$ is of type ${\rm III_1}$ (see \cite{Sh96}). Since $\core(N)$ is a factor and since $\core(\sigma^\pi)_h$ is outer for every $h \in \FC(G) \setminus g^\Z$ (by combining \cite[Proposition 5.4]{HS88}, Lemma \ref{lem-outer} and item $(\rm ii)$), it follows that 
$$\core(N)' \cap ((\core(N) \rtimes \FC(G)) = \core(N)' \cap (\core(N) \rtimes g^\Z).$$ 
Using the above observation with Claim \ref{claim-2}, we infer that 
\begin{align*}
\mathcal Z(\core(M)) &= \rL(G)' \cap \left( \core(N)' \cap (\core(N) \rtimes \FC(G)) \right) \\
&= \rL(G)' \cap \left( \core(N)' \cap (\core(N) \rtimes g^\Z) \right) \\
&= \rL(G)' \cap \rL(g^\Z) \quad (\text{since } \core(N) \text{ is a factor})\\
&= \rL(g^\Z) = \rL^\infty(\R/T\Z).
\end{align*}
Since the action $\R \curvearrowright \R/T\Z$ is essentially transitive, the dual action $\R \curvearrowright \mathcal Z(\core(M))$ is essentially transitive and so $M$ is of type ${\rm III_\lambda}$.

Secondly, assume that $\kappa \geq 1$. We have $1 = U_{\kappa T} = \pi_{g^\kappa}$. Since $\pi|_{\FC(G)}$ is faithful, we have $g^\kappa = 1$ and so have $g^\Z \cong \Z/\kappa\Z$. Since $\rT(N) = \kappa T  \Z= \frac{2 \pi}{\log \lambda^{1/\kappa}} \Z$, $N$ is of type ${\rm III_{\lambda^{1/\kappa}}}$ (see \cite{Sh96}). Then $\varphi$ is $\kappa T$-periodic and we have $(N_\varphi)' \cap N = \C1$ by \cite[Th\'eor\`eme 4.2.6]{Co72}. Observe that $M_\psi = N_\varphi \rtimes G$. We next prove the following claim.

\begin{claim}\label{claim-3}
$\mathcal Z(M_\psi) = \rL(g^\Z)$.
\end{claim}

Indeed, by Claim \ref{claim-1}, we have 
$$\mathcal Z(M_\psi) = (N_\varphi)' \cap \rL(G)' \cap (N_\varphi \rtimes G) = (N_\varphi)' \cap \rL(G)' \cap (N_\varphi \rtimes \FC(G)).$$
 Let $x \in (N_\varphi)' \cap (N_\varphi \rtimes \FC(G))$ and write $x = \sum_h x^h u_h$ for its Fourier decomposition. Then we have $y x^h = x^h \sigma^\pi_h(y)$ for every $h \in \FC(G)$ and every $y \in N_\varphi$. Since $(N_\varphi)' \cap N = \C 1$, for every $h \in \FC(G)$, there exist $\alpha_h \in \C$ and $v_h \in \mathcal U(N_\varphi)$ so that $x^h = \alpha_h v_h$. Let $h \in \FC(G)$ such that $\alpha_h \neq 0$. Then $\sigma^\pi_h |_{N_\varphi} = \Ad(v_h)|_{N_\varphi}$. Then $\Ad(v_h^*) \circ \sigma^\pi_h$ is $\varphi$-preserving and $\Ad(v_h^*) \circ \sigma^\pi_h|_{N_\varphi} = \id_{N_\varphi}$. Applying the proof of \cite[Theorem 3.2]{HS88} to the periodic weight $\varphi \otimes \Tr_{\mathbf B(\ell^2)}$, there exists $t \in \R$ such that $\Ad(v_h^*) \circ \sigma^\pi_h = \sigma_t^\varphi$. Lemma \ref{lem-outer} implies that $U_t = \pi_h$ and so $t \in \rT(M)$ and $h \in g^\Z$. This shows that $(N_\varphi)' \cap (N_\varphi \rtimes \FC(G)) = (N_\varphi)' \cap (N_\varphi \rtimes g^\Z) = \rL(g^\Z)$ (since $N_\varphi$ is a factor) and so $\mathcal Z(M_\psi) = \rL(g^\Z)$. 

Since $\mathcal Z(M_\psi)$ is discrete, $M$ cannot be of type ${\rm III_0}$ (see \cite[Corollaire 3.2.7(b)]{Co72}) and so $M$ is of type ${\rm III_\lambda}$ by \cite[Th\'eor\`eme 3.4.1]{Co72}. 

If $M$ is of type ${\rm II_1}$, then $\rT(M) = \R$. Conversely, if $\rT(M) = \R$, since $M$ has separable predual, it follows that $M$ is semifinite by \cite[Th\'eor\`eme 1.3.4 (b)]{Co72}. Since $N \subset M$ is with expectation, this further implies that $N$ is semifinite (see \cite[Lemma ${\rm V}$.2.29]{Ta02}) and so $N$ is a type ${\rm II_1}$ factor by \cite{Sh96}. Since the free Bogoljubov action $G \curvearrowright (N, \tau)$ is trace-preserving, $M = N \rtimes G$ is a type ${\rm II_1}$ factor. 

Using the above reasoning, $M$ is of type ${\rm III_0}$ if and only if $\rT(M)$ is dense in $\R$ and $\rT(M) \neq \R$. Therefore, Connes' invariant $\rT(M)$ completely determines the type of $M$.

$(\rm iv)$ Finally, assume that $M$ is a type ${\rm III_1}$ factor. For bicentralizer algebras, we use notation of \cite[Section 3]{HI15}. To prove that $M$ has trivial bicentralizer, it suffices to show that $\rB(M, \psi) = \C 1$. Using item $({\rm ii})$, $Q = N \rtimes \FC(G)$ is a type ${\rm III_1}$ factor. Choose a nonprincipal ultrafilter $\omega \in \beta(\N)\setminus \N$. Using \cite[Proposition 3.2]{HI15} and Claim \ref{claim-1}, we have 
$$\rB(M, \psi) = ((M^\omega)_{\psi^\omega})' \cap M \subset ((Q^\omega)_{\psi^\omega})' \cap \rL(G)' \cap M  \subset ((Q^\omega)_{\psi^\omega})' \cap Q =  \rB(Q, \psi).$$
In order to prove that $\rB(M, \psi) = \C 1$, it suffices to show that $\rB(Q, \psi) = \C 1$. 

Put $H = \FC(G)$. Fix an enumeration $\{h_n \mid n \in \N\}$ of $H$. For every $n \in \N$, denote by $H_n$ the subgroup of $H$ generated by $\{h_0, \dots, h_n\}$.  Then $(H_n)_n$ is an increasing sequence of subgroups of $H$ such that $\bigcup_{n \in \N} H_n = H$. For every $n \in \N$, since $H_n$ is finitely generated and since $\FC(H_n) = H_n$, its center $\mathcal Z(H_n)$ has finite index in $H_n$ and so $H_n$ is virtually abelian. Using item $(\rm ii)$, $Q_n = N \rtimes H_n$ is a type ${\rm III_1}$ factor. Theorem \ref{thm-prime} implies that $Q_n$ is semisolid and \cite[Theorem 3.7]{HI15} implies that $Q_n$ has trivial bicentralizer. Thus, we have $\rB(Q_n, \psi) = \C 1$ for every $n \in \N$. 

For every $n \in \N$, denote by $\rE_{Q_n} : Q \to Q_n$ the unique $\psi$-preserving conditional expectation. Let $x \in \rB(Q, \psi)$. Then $\rE_{Q_n}(x) \in \rB(Q_n, \psi)$ and so $\rE_{Q_n}(x) = \psi(x)1$. Since $\lim_n \|x - \rE_{Q_n}(x)\|_\psi = 0$, it follows that $x = \psi(x) 1$. Thus, we have $\rB(Q, \psi) = \C 1$ and so $\rB(M, \psi) = \C1$. This finishes the proof of Theorem \ref{letterthm-type}.
\end{proof}

\section{Fullness}

We fix the following notation. Let $U : \R \curvearrowright H_\R$ be any strongly continuous orthogonal representation with $\dim H_\R \geq 2$. Let $G$ be any countable group and $\pi : G \curvearrowright H_\R$ any faithful orthogonal representation such that $[U, \pi] = 0$. Put $M = \Gamma(U, \pi)\dpr$ and denote by $(M, \rL^2(M), J, \rL^2(M)_+)$ its standard form. Write $\varphi \in M_\ast$ for the canonical faithful state on $M$. 

Write $H_\R = H_\R^{\ap} \oplus H_\R^{\wm}$ where $H_\R^{\ap}$ (resp.\ $H_\R^{\wm}$) is the {\em almost periodic} (resp.\ {\em weakly mixing}) part of the orthogonal representation $\pi : G \curvearrowright H_\R$. Likewise, write $H = H^{\ap} \oplus H^{\wm}$ where $H^{\ap}$ (resp.\ $H^{\wm}$) is the almost periodic (resp.\ weakly mixing) part of the corresponding unitary representation $\pi : G \curvearrowright H$. As usual, denote by $A$ the infinitesimal generator of $U : \R \curvearrowright H$ and let $j : H_\R \to H : \eta \mapsto ( \frac{2}{ A^{-1} + 1})^{1/2} \eta$ be the corresponding isometric embedding. Put $K_\R = j(H_\R)$, $K_\R^{\ap} = j(H_\R^{\ap})$ and $K_\R^{\wm} = j(H_\R^{\wm})$. Recall that $K_\R \cap {\rm i} K_\R = \{0\}$ and $K_\R + {\rm i} K_\R$ is dense in $H$. We consider the involution $K_\R + {\rm i} K_\R \to K_\R + {\rm i} K_\R$ given $\overline{\xi + {\rm i} \eta} = \xi - {\rm i} \eta$ for all $\xi, \eta \in K_\R$.

Let  $L \subset K_\R + {\rm i} K_\R$ be any subspace such that $\overline L=L$. Denote by $\mathcal{X}(L)$ the closure in $\rL^2(M)$ of the linear span of all the elements of the form $e_1\otimes \cdots\otimes e_k \otimes \delta_h$ where $k \geq 1$, $e_1\in L$, $e_2, \dots, e_k \in K_\R + {\rm i} K_\R$, $h \in G$.

\begin{lem}\label{lem-vanishing}
 Assume that $H_\R^{\wm} \neq 0$. Let $L \subset K_\R^{\wm} + {\rm i} K_\R^{\wm}$ be any nonzero finite dimensional subspace. Then for every $x = (x_n)^\omega\in \rL(G)' \cap M^\omega$, we have $\lim_{n \to \omega}\|P_{\cX(L)}(x_n\xi_\varphi)\|=0$.
\end{lem}

\begin{proof}
Fix an integer $N \geq 1$. Since the representation $\pi|_{H^{\wm}}$ is weakly mixing and since $L \subset K_\R^{\wm} + {\rm i} K_\R^{\wm} \subset H^{\wm}$ is finite dimensional, we may choose inductively elements $g_1,\dots, g_N \in G$ such that $\pi_{g_i}(L) \perp_{1/(N\dim(L))} \pi_{g_j}(L)$ for all $1\leq i<j\leq N$. We start by proving the following claim.
\begin{claim}\label{claim-epsilon}
For all $1 \leq i < j \leq N$, we have $\cX(\pi_{g_i}(L)) \perp_{1/N} \cX(\pi_{g_j}(L))$.
\end{claim}

Indeed, let $1 \leq i < j \leq N$ and put $g=g_j^{-1}g_i \in G$. Put $U_g = u_g Ju_gJ \in \mathcal U(\rL^2(M))$. We have $U_g(\cX(L))=\cX(\pi_g(L))$.  Let $(l_a)_{1 \leq a \leq \dim(L)}$ be an orthonormal basis for the space $L$. Let $\xi, \eta \in \cX(L)$ be any elements. Observe that we may and will identify $\rL^2(M) \ominus \rL^2(\rL(G))$ with $H \otimes \rL^2(M)$ via the unitary operator 
$$U : \rL^2(M) \ominus \rL^2(\rL(G)) \to H \otimes \rL^2(M) : \zeta_1 \otimes \cdots \otimes \zeta_\ell \otimes \delta_g \mapsto \zeta_1 \otimes (\zeta_2 \otimes \cdots \otimes \zeta_\ell \otimes \delta_g)$$
where $\ell \geq 1$, $\zeta_1, \dots, \zeta_\ell \in H$, $g \in G$. We can then write $\xi=\sum_{a=1}^{\dim(L)} l_a\otimes \xi_a$ and $\eta=\sum_{b=1}^{\dim(L)} l_b\otimes \eta_b$ for some $\xi_a, \eta_b \in \rL^2(M)$. Using Cauchy--Schwarz inequality and the assumption that $\pi_g(L)\perp_{(1/N \dim(L))} L$, we have
\begin{align*}
|\langle U_g( \xi),\eta\rangle|&\leq\sum_{a,b = 1}^{\dim(L)}|\langle\pi_g(l_a),l_b\rangle|\ |\langle U_g(\xi_a),\eta_b\rangle| \\
&\leq \frac{1}{N\dim(L)}\sum_{a,b = 1}^{\dim(L)} \|\xi_a\|\ \|\eta_b\| \\
&\leq \frac{1}{N}\|\xi\|\ \|\eta\|.
\end{align*}
This proves that $\cX(\pi_g(L))\perp_{1/N}\cX(L)$ and so $\cX(\pi_{g_i}(L)) \perp_{1/N} \cX(\pi_{g_j}(L))$.

For all $g\in G$ and all $n\in \N$, we have
\begin{align*}
  \|P_{\cX(L)}(x_n\xi_\varphi)\|^2&=\langle P_{\cX(L)}(x_n\xi_\varphi),x_n\xi_\varphi\rangle \\
  &= \langle U_g(P_{\cX(L)}(x_n\xi_\varphi)),U_g(x_n\xi_\varphi)\rangle \\
  &=\langle P_{\cX(\pi_g(L))}(U_g(x_n\xi_\varphi)),U_g(x_n\xi_\varphi)\rangle
\end{align*}
Since $u_g \in \rL(G) \subset M_\varphi$, we have $U_g(x_n\xi_\varphi) =u_gx_nu_g^\ast\xi_\varphi$. Since $x = (x_n)^\omega\in \rL(G)' \cap  M^\omega$, we have $\lim_{n\to\omega}\|(u_gx_nu_g^\ast - x_n)\xi_\varphi\|=0$ and so
\[\lim_{n\to\omega} \|P_{\cX(L)}(x_n\xi_\varphi)\|^2=\lim_{n\rightarrow\omega} \langle P_{\cX(\pi_g(L))}(x_n\xi_\varphi),x_n\xi_\varphi\rangle.\]
Applying the above result to $g_1, \dots g_N \in G$, we obtain, using Cauchy--Schwarz inequality,
\begin{align*}
    \lim_{n\to\omega} \|P_{\mathcal X(L)}(x_n\xi_\varphi)\|^2&=\lim_{n\rightarrow\omega} \frac1N \sum_{i=1}^N \langle P_{\cX(\pi_{g_i}(L))}(x_n\xi_\varphi),x_n\xi_\varphi\rangle \\
 &=\lim_{n\to\omega} \frac1N\left\langle \sum_{i=1}^N P_{\cX(\pi_{g_i}(L))}(x_n\xi_\varphi),x_n\xi_\varphi\right\rangle \\
 &\leq \lim_{n\to\omega} \frac1N \left\Vert\sum_{i=1}^N P_{\cX(\pi_{g_i}(L))}(x_n\xi_\varphi)\right\Vert\|x_n\|_\varphi.
\end{align*}
Using Claim \ref{claim-epsilon}, for all $n \in \N$, we have
\begin{align*}
  \left\Vert\sum_{i=1}^N P_{\mathcal X(\pi_{g_i}(L))}(x_n \xi_\varphi)\right\Vert^2&= \sum_{1\leq i,j\leq N} \langle P_{\mathcal X(\pi_{g_i}(L))}(x_n \xi_\varphi),P_{\mathcal X(\pi_{g_i}(L))}(x_n \xi_\varphi)\rangle \\
  &\leq \sum_{i=1}^N \|P_{\mathcal X(\pi_{g_i}(L))}(x_n \xi_\varphi)\|^2 + \frac1N \sum_{1 \leq i\neq j \leq N}\|x_n\|_\varphi^2 \\
  &\leq N\|x_n\|_\varphi^2 + \frac{N(N - 1)}{N}\|x_n\|_\varphi^2 = (2N - 1)\|x_n\|_\varphi^2.
\end{align*}

In the end, we obtain $\lim_{n\to \omega} \|P_{\cX(L)}(x_n\xi_\varphi)\|^2\leq \frac{\sqrt{2 N - 1}}{N} \|x\|_{\varphi^\omega}^2$. Since this is true for every $N \geq 1$, we obtain that $\lim_{n\to \omega} \|P_{\cX(L)}(x_n\xi_\varphi)\| = 0$.
\end{proof}

Following \cite[Appendix C]{HI15}, define $K_{\an}^{\wm} = \bigcup_{\lambda > 0} \mathbf 1_{[\lambda^{-1}, \lambda]}(A)(H^{\wm}) \subset K_\R^{\wm} + {\rm i} K_\R^{\wm}$. Observe that $\overline{K_{\an}^{\wm}} = K_{\an}^{\wm}$ and that $K_{\an}^{\wm} \subset K_\R^{\wm} + {\rm i} K_\R^{\wm}$ is a dense subspace in $H^{\wm}$ of  elements $\eta \in K_\R^{\wm} + {\rm i} K_\R^{\wm}$ for which the map $\R \mapsto K_\R^{\wm} + {\rm i} K_\R^{\wm} : t \mapsto U_t \eta$ extends to a $(K_\R^{\wm} + {\rm i} K_\R^{\wm})$-valued entire analytic map. For all $\eta \in K_{\an}^{\wm}$, the element $W(\eta)$ is entire analytic for the modular automorphism group $\sigma^\varphi$ and we have $\sigma_z^\varphi(W(\eta)) = W(A^{{\rm i} z}\eta)$ for every $z \in \C$.

\begin{lem}\label{lem-key}
 Assume that $H_\R^{\wm} \neq 0$. Let $x = (x_n)^\omega\in \rL(G)' \cap M^\omega$, $\xi \in K_{\an}^{\wm}$ any unit vector and $(t_n)_{n\in \N}$ any real-valued sequence such that
  \[ \lim_{n \to \omega} \|x_nW(U_{t_n}\xi)- W(\xi)x_n \|_\varphi=0.\]
  Then we have $x\in \rL(G)^\omega$.
\end{lem}

\begin{proof}
Observe that for every $t \in \R$, we have $\sigma^\varphi_{-{\rm i}/2}(W(U_t \xi)^*) = \sigma^\varphi_{-{\rm i}/2}(W(U_t \overline \xi)) = W(A^{1/2 + {\rm i}t} \overline \xi)$. For every $n \in \N$, put $y_n=x_n-\rE_{\rL(G)}(x_n) \in M \ominus \rL(G)$. Put $y = x - \rE_{\rL(G)^\omega}(x) \in M^\omega \ominus \rL(G)^\omega$ so that $y = (y_n)^\omega$. For all $n\in\N$, we have
\begin{align}\label{eq-one}
 (x_n W(U_{t_n} \xi) - W(\xi) x_n)\xi_\varphi &=  (y_n W(U_{t_n}\xi) -  W(\xi)y_n) \xi_\varphi \\ \nonumber
 & \quad + (\rE_{\rL(G)}(x_n) W(U_{t_n}\xi) -  W(\xi)\rE_{\rL(G)}(x_n)) \xi_\varphi \\ \nonumber
 &= J W(A^{1/2 + {\rm i} t_n}\overline \xi) J y_n \xi_\varphi -  W(\xi)y_n \xi_\varphi \\ \nonumber
 & \quad + \rE_{\rL(G)}(x_n) W(U_{t_n}\xi)\xi_\varphi -  W(\xi)\rE_{\rL(G)}(x_n) \xi_\varphi.
\end{align}
Put $L = \spn \{\xi, \overline\xi\} \subset K_\R^{\wm} + {\rm i} K_\R^{\wm}$. Lemma \ref{lem-vanishing} implies that 
  \[ \lim_{n \to\omega} \|y_n\xi_\varphi - P_{\cX(L^\perp)}(y_n\xi_\varphi)\|_2=0.\]
Since $(JW(A^{1/2 + {\rm i} t_n}\overline \xi) J)_n$ is uniformly bounded, in the ultraproduct Hilbert space $\rL^2(M)^\omega$, we have the following equalities
\begin{align}\label{eq-two}
(W(\xi)y_n \xi_\varphi)_\omega &= (W(\xi) P_{\cX(L^\perp)}(y_n \xi_\varphi))_\omega \\ \nonumber
(J W(A^{1/2 + {\rm i} t_n}\overline \xi) J y_n \xi_\varphi)_\omega &=  (J W(A^{1/2 + {\rm i} t_n}\overline \xi) J P_{\cX(L^\perp)}(y_n \xi_\varphi))_\omega.
\end{align}
Using the Wick formula, for every $n \in \N$, we have that
\begin{itemize}
\item $W(\xi) P_{\cX(L^\perp)}(y_n \xi_\varphi)$ lies in the closure in $\rL^2(M)$ of the linear span of all the elements of the form $\xi \otimes \eta \otimes e_1 \otimes \cdots \otimes e_k \otimes \delta_h$ where $k \geq 0$, $\eta \in L^\perp$, $e_1, \dots, e_k \in K_\R + {\rm i} K_\R$, $h \in G$.
\item $J W(A^{1/2 + {\rm i} t_n}\overline \xi) J P_{\cX(L^\perp)}(y_n \xi_\varphi)$  lies in the closure in $\rL^2(M)$ of the linear span of all the elements of the form $\delta_h$ or $\eta \otimes e_1 \otimes \cdots \otimes e_k \otimes \delta_h$ where $k \geq 0$, $\eta \in L^\perp$, $e_1, \dots, e_k \in K_\R + {\rm i} K_\R$, $h \in G$.
\item $\rE_{\rL(G)}(x_n) W(U_{t_n}\xi)\xi_\varphi -  W(\xi)\rE_{\rL(G)}(x_n) \xi_\varphi$ lies in the closure in $\rL^2(M)$ of the linear span of all the elements of the form $\delta_h$ or $e \otimes \delta_h$ where $e \in K_\R +{\rm i} K_\R$, $h \in G$.
\end{itemize}
This implies that $W(\xi) P_{\cX(L^\perp)}(y_n \xi_\varphi)$ is orthogonal to $J W(A^{1/2 + {\rm i} t_n}\overline \xi) J P_{\cX(L^\perp)}(y_n \xi_\varphi)$ and to  $\rE_{\rL(G)}(x_n) W(U_{t_n}\xi)\xi_\varphi -  W(\xi)\rE_{\rL(G)}(x_n) \xi_\varphi$ for every $n \in \N$. Using moreover the assumption together with  \eqref{eq-one} and \eqref{eq-two}, we obtain $\lim_{n \to \omega} \|W(\xi) P_{\cX(L^\perp)}(y_n \xi_\varphi)\| = 0$. Since for every $n \in \N$, we have $W(\xi) P_{\cX(L^\perp)}(y_n \xi_\varphi) = \xi \otimes P_{\cX(L^\perp)}(y_n \xi_\varphi)$, it follows that $\lim_{n \to \omega} \|P_{\cX(L^\perp)}(y_n \xi_\varphi)\| = 0$ and so $\lim_{n \to \omega} \|y_n \xi_\varphi\| = 0$. This shows that $x - \rE_{\rL(G)^{\omega}}(x) = y = 0$ and so $x = \rE_{\rL(G)^{\omega}}(x) \in \rL(G)^\omega$.
\end{proof}

\begin{proof}[Proof of Theorem \ref{letterthm-full}]
Put $M = \Gamma(U, \pi)\dpr = N \rtimes G$. 

$(\rm i)$ We first consider the case when $G$ is finite. In that case, the image of $\sigma^\pi(G)$ in $\Out(N)$ is finite thus discrete. Then \cite[Theorem B]{Ma16} implies that $M = N \rtimes G$ is a full factor.

We next consider the case when $G$ is infinite. Since $\pi(G) \subset \mathcal O(H_\R)$ is infinite ($\pi$ is faithful) and discrete with respect to the strong topology, it follows that the weakly mixing part of $\pi$ is nonzero, that is, $H_\R^{\wm} \neq 0$. We may then choose a unit vector $\xi \in K_{\an}^{\wm}$.

Let $x = (x_n)^\omega \in M' \cap M^\omega$ be any element. Since $x = (x_n)^\omega \in \rL(G)' \cap M^\omega$ and since $\lim_{n \to \omega} \|x_n W(\xi) - W(\xi) x_n\|_\varphi = 0$, Lemma \ref{lem-key} implies that $x \in \rL(G)^\omega$. We may then replace each $x_n$ by $\rE_{\rL(G)}(x_n)$ and assume that $x = (x_n)^\omega$ where $x_n \in \rL(G)$. Since $\pi(G)\subset O(H_\R)$ is discrete with respect to the strong topology, $\sigma^\pi(G)\subset \Aut(M)$ is discrete with respect to the $u$-topology. Since $\pi$ is faithful and since $\sigma^\pi$ is $\varphi$-preserving, there exist $\kappa > 0$ and $y_1,\dots, y_m\in N$ such that for all $g\in G \setminus \{0\}$, we have
\[ \sum_{k = 1}^m \| \sigma^\pi_g(y_k)-y_k\|^2_\varphi\geq \kappa.\]
 Write $x_n=\sum_{g\in G} (x_n)^g u_g$ for the Fourier decomposition of $x_n$ in $\rL(G)$. We have
 \begin{align*}
\sum_{k=1}^m \| x_ny_k-y_kx_n\|^2_\varphi &=  \sum_{g\in G} |(x_n)^g|^2\sum_{k=1}^m\| \sigma^\pi_g(y_k)-y_k\|^2_\varphi \\
&\geq \kappa\sum_{g\in G \setminus \{e\}} |(x_n)^g|^2 \\
&=\kappa\|x_n-\tau(x_n)1\|^2_2.
 \end{align*}
Since $x \in M' \cap \rL(G)^\omega$, we have $\lim_{n \to \omega} \|x_ny_k-y_kx_n\|_\varphi=0$ for all $1\leq k\leq m$. This implies that $\lim_{n \to \omega}\| x_n -\tau(x_n)1\|_2=0$ and so $x \in \C 1$. This shows that $M$ is full.

$(\rm ii)$ Assume that $(t_n)_{n\in\N}$ is a sequence converging to $0$ with respect to $\tau(M)$. By definition, it means that the class of $\sigma^\varphi_{t_n}$ converges to $1$ in $\Out(M)$. Therefore, there exists a sequence of unitaries $u_n\in \cU(M)$ such that $\Ad(u_n)\circ \sigma^\varphi_{t_n} \to  \id_M$ with respect to the $u$-topology in $\Aut(M)$. Since $\Aut(M)$ is a topological group, $\sigma^\varphi_{-t_n}\circ \Ad(u_n^*) = (\Ad(u_n)\circ \sigma^\varphi_{t_n})^{-1}\to \id_M$ with respect to the $u$-topology in $\Aut(M)$. This implies that $\lim_n \|u_n \varphi - \varphi u_n\| = 0$. Let $\omega \in \beta(\N) \setminus \N$ be any nonprincipal ultrafilter. We have $(u_n)_n \in \mathfrak M^\omega(M)$ and $(u_n)^\omega\in (M^\omega)_{\varphi^\omega}$.

For every $g\in G$, we have $\lim_{n \to \omega} \|u_n\sigma_{t_n}^\varphi(u_g)-u_g u_n \|_\varphi =0$. Since $\sigma^\varphi_{t_n}(u_g)=u_g$, we obtain that $(u_n)^\omega \in \rL(G)' \cap M^\omega$. We also have $\lim_{n \to \omega} \|u_n W(U_{t_n}\xi)- W(\xi)u_n\|_\varphi  = \lim_{n \to \omega} \|u_n \sigma_{t_n}^\varphi(W(\xi))- W(\xi)u_n\|_\varphi = 0$. Lemma \ref{lem-key} implies that
$(u_n)^\omega\in \rL(G)^\omega$. Since $u = (u_n)^\omega \in \rL(G)^\omega$, we may choose unitaries  $v_n \in \mathcal U(\rL(G))$ such that $u  = (v_n)^\omega$. We then have $\Ad(v_n^* u_n) \to \id_M$ as $n \to \omega$ and so $\Ad(v_n) \circ \sigma_{t_n}^\varphi \to \id_M$ as $n \to \omega$ with respect to the $u$-topology in $\Aut(M)$. 

Write $v_n=\sum_{g\in G}(v_n)^gu_g$ for the Fourier decomposition of $v_n$ in $\rL(G)$. We claim that 
\begin{equation}\label{eq-gap}
\exists \kappa > 0, \exists y_1,\dots,y_m\in N, \forall g\in G\setminus\{e\}, \forall t \in \R, \quad \sum_{k=1}^m \|(\sigma^\pi_g \circ \sigma^\varphi_t)(y_k)-(y_k)\|^2_\varphi\geq \kappa.
\end{equation}
Indeed, if \eqref{eq-gap} does not hold, then there exist $g_i \in G \setminus \{e\}$ and $t_i \in \R$ such that $\sigma^\pi_{g_i}\circ \sigma^\varphi_{t_i}\to \id_N$ with respect to the $u$-topology in $\Aut(N)$. This implies that $\rho_{(t_i, g_i)} = U_{t_i} \pi_{g_i}\to 1$ strongly, contradicting the assumption on $\rho$. Then \eqref{eq-gap} holds. For every $n \in \N$, we have
\begin{align*}
\sum_{k=1}^m \| v_n \sigma^\varphi_{t_n}(y_k) - y_k v_n\|^2_\varphi &=  \sum_{g\in G} |(v_n)^g|^2\sum_{k=1}^m\| (\sigma^\pi_g \circ \sigma^\varphi_{t_n})(y_k) - y_k\|^2_\varphi \\
&\geq \kappa\sum_{g\in G \setminus \{e\}} |(v_n)^g|^2 \\
&=\kappa\|v_n-\tau(v_n)1\|^2_2.
\end{align*}
Since $\Ad(v_n) \circ \sigma_{t_n}^\varphi \to \id_M$ as $n \to \omega$ with respect to the $u$-topology in $\Aut(M)$, we have $\lim_{n \to \omega} \|v_n\sigma_{t_n}^\varphi(y_k)-y_k v_n\|_\varphi=0$ for all $1\leq k\leq m$. This implies that $\lim_{n \to \omega}\| v_n -\tau(v_n)1\|_2=0$ and so $\lim_{n \to \omega}\| u_n -\varphi(u_n)1\|_\varphi=0$ . Since this is true for every nonprincipal ultrafilter $\omega \in \beta(\N) \setminus \N$, we have $\lim_n \|u_n - \varphi(u_n)1\|_\varphi = 0$ and so $\Ad(u_n) \to \id_M$ with respect to the $u$-topology in $\Aut(M)$. Therefore $\sigma_{t_n}^\varphi\to \id$ with respect to the $u$-topology in $\Aut(M)$. This further implies that $U_{t_n}\to 1$ strongly which means that $t_n\rightarrow 0$ with respect to $\tau(U)$.
\end{proof}

\section{Amenable and Gamma absorption}

Before proving Theorem \ref{letterthm-maximal}, we state a type ${\rm III}$ version of Krogager--Vaes' result \cite[Theorem 5.1 (2)]{KV16}.

\begin{thm}\label{thm-maximal}
Let $U : \R \curvearrowright H_\R$ be any strongly continuous orthogonal representation. Let $G$ be any  countable group and $\pi : G \curvearrowright H_\R$ any orthogonal representation such that $[U, \pi] = 0$. Put $M = \Gamma(U, \pi)\dpr$.
\begin{itemize}
\item [$(\rm i)$] Assume that $\pi : G \curvearrowright H_\R$ is weakly mixing. Let $\rL(G) \subset P \subset M$ be any intermediate von Neumann subalgebra with expectation such that $P \lessdot_M \rL(G)$. Then $P = \rL(G)$.

\item [$(\rm ii)$] Assume that $\pi : G \curvearrowright H_\R$ is mixing. Let $P \subset M$ be any von Neumann subalgebra with expectation such that $P \lessdot_M \rL(G)$ and $P \cap \rL(G)$ is diffuse. Then $P \subset \rL(G)$.
\end{itemize}
\end{thm}

\begin{proof}
$(\rm i)$ We adapt the proof of \cite[Theorem 5.1 $(2)$]{KV16} for the reader's convenience. Put $M = \Gamma(U, \pi)\dpr$ and denote by $\varphi$ the canonical faithful normal state on $M$. Since $P \lessdot_M \rL(G)$, there exists a norm one projection $\Phi : \langle M, \rL(G)\rangle \to P$ such that $\Phi|_M : M \to P$ is the unique $\varphi$-preserving conditional expectation. Actually there exists a unique faithful normal conditional expectation $\rE_{P} : M \to P$. Indeed, since $\pi$ is weakly mixing, the $\varphi$-preserving action $G \curvearrowright (N, \varphi)$ is $\varphi$-weakly mixing. Theorem \ref{thm-mixing} $(\rm i)$ implies that $P' \cap M \subset \rL(G)' \cap M \subset \rL(G) \subset P$. Then \cite[Th\'eor\`eme 1.5.5]{Co72} implies that there exists a unique faithful normal conditional expectation $\rE_{P} : M \to P$.

As usual, denote by $A$ the infinitesimal generator of the strongly continuous unitary representation $U : \R \curvearrowright H$. Put $K_\R = j(H_\R)$ where $j : H_\R \to H : \xi \mapsto ( \frac{2}{ A^{-1} + 1})^{1/2}\xi$ is the canonical isometric embedding. Recall that $K_\R \cap {\rm i} K_\R = \{0\}$ and $K_\R +{\rm i} K_\R$ is dense in $H$. 

Denote by $(M, \rL^2(M), J, \rL^2(M)_+)$ the standard form of $M$. We identify $\rL^2(M) = \mathcal F(H) \otimes \ell^2(G)$ and we view $M = N \rtimes G \subset \mathbf B(\mathcal F(H) \otimes \ell^2(G))$ as generated by $b \otimes 1$ for $b \in N$ and by $u_g = \mathcal F(\pi_g) \otimes \lambda_g$ for $g \in G$. For every $g \in G$, we have $Ju_gJ = 1 \otimes \rho_{g}$. For every $T \in \mathbf B(\mathcal F(H) )$, we identify $T$ with $T \otimes 1 \in \mathbf B(\mathcal F(H) \otimes \ell^2(G))$. In particular, for every $T \in  \mathbf B(\mathcal F(H) )$, we have $T \in (J \rL(G) J)' \cap \mathbf B(\rL^2(M)) = \langle M, \rL(G)\rangle$.

For every unit vector $e \in K_\R$, denote by $P_e = \ell(e)\ell(e)^*$ the orthogonal projection $\rL^2(M) \to \mathcal X(\C e)$ where $\mathcal X(\C e)$ is the closure in $\rL^2(M)$ of the linear span of all the elements of the form $e \otimes e_1 \cdots \otimes e_k \otimes \delta_h$ where $k \geq 0$, $e_1, \dots , e_k \in K_\R +{\rm i} K_\R$, $h \in G$.

\begin{claim}\label{claim-vanishing}
For every unit vector $e \in K_\R$, we have $\Phi(P_e) = 0$.
\end{claim}

Indeed, fix an integer $N \geq 1$. Since $\pi$ is weakly mixing, we may choose inductively elements $g_1,\dots, g_N \in G$ such that $|\langle \pi_{g_i}(e) , \pi_{g_j}(e)\rangle | \leq 1/N$ for all $1\leq i<j\leq N$.  Claim \ref{claim-epsilon} implies that $\mathcal X( \C\pi_{g_i}(e)) \perp_{1/N} \mathcal X( \C \pi_{g_j}(e))$ and so $\|P_{\pi_{g_i}(e)} P_{\pi_{g_j}(e)}\|_\infty \leq 1/N$ for all $1\leq i<j\leq N$. Put $\phi = \varphi \circ \Phi \in \langle M, \rL(G) \rangle^*$. Since $u_{g_i} \in \rL(G) \subset M_\varphi$ for all $1 \leq i \leq N$, we have 
$$\phi(P_e) = \frac1N \sum_{i = 1}^N \phi(u_{g_i}P_{e}u_{g_i}^*) =\frac1N \sum_{i = 1}^N \phi(P_{\pi_{g_i}(e)}) = \frac1N \phi\left(\sum_{i = 1}^N P_{\pi_{g_i}(e)} \right) \leq \frac1N \left \| \sum_{i = 1}^N P_{\pi_{g_i}(e)} \right \|_\phi.$$
Moreover, we have
\begin{align*}
\left \| \sum_{i = 1}^N P_{\pi_{g_i}(e)} \right \|_\phi^2 &= \sum_{1 \leq i , j \leq N} \langle P_{\pi_{g_i}(e)} , P_{\pi_{g_j}(e)} \rangle_\phi \\
&\leq  \sum_{i = 1}^N \|P_{\pi_{g_i}(e)} \|_\phi^2 + \sum_{1 \leq i \neq j \leq N} \|P_{\pi_{g_i}(e)} P_{\pi_{g_j}(e)}\|_\infty \\
& \leq N + \frac{N(N - 1)}{N} = 2N - 1.
\end{align*}
This implies that $\phi(P_e) \leq \frac{\sqrt{2N - 1}}{N}$. Since this is true for every $N \geq 1$, we infer that $\phi(P_e) = 0$ and so $\Phi(P_e) = 0$ since $\varphi$ is faithful.

Let $e \in K_\R$ and $T \in \langle M, \rL(G)\rangle$. Applying Kadison's inequality and Claim \ref{claim-vanishing}, we have
$$\Phi(\ell(e)T)\Phi(\ell(e)T)^* \leq \Phi(\ell(e) TT^* \ell(e)^*) \leq \|T\|^2 \, \Phi(\ell(e) \ell(e)^*) = 0$$
and so $\Phi(\ell(e)T) = 0$. Likewise, we have $\Phi(T \ell(e)^*) = 0$. Using Wick formula, we obtain $\Phi(W(e_1 \otimes \cdots \otimes e_k)) = 0$ for all $k \geq 1$ and all $e_1, \dots, e_k \in K_\R + {\rm i} K_\R$. This further implies that $\Phi (M \ominus \rL(G)) = 0$ and so $P = \rL(G)$.

$(\rm ii)$ Along the lines of the proof of item $(\rm i)$, the proof of \cite[Theorem 5.1 $(2)$]{KV16} for $A = \rL(G)$ and $B = P$ applies {\em mutatis mutandis} since $P \cap \rL(G)$ is tracial. We obtain that $P \subset \rL(G)$.
\end{proof}

\begin{proof}[Proof of Theorem \ref{letterthm-maximal}]
$(\rm i)$ This follows from item $(\rm i)$ in Theorem \ref{thm-maximal}. 

$(\rm ii)$ Let $P \subset M$ be any von Neumann subalgebra with expectation and with property Gamma such that $P \cap \rL(G)$ is diffuse. Since $P \subset M$ is with expectation and has property Gamma, \cite[Theorem 3.1]{HU15b} implies that there exists a faithful state $\psi \in M_\ast$ such that $P$ is globally invariant under $\sigma^\psi$ and there exists a decreasing sequence of diffuse abelian von Neumann subalgebras $A_k \subset P_\psi$ such that $\bigvee_{k \in \N} ((A_k)' \cap P) = P$. Observe that $Q_k = (A_k)' \cap P$ is globally invariant under $\sigma^\psi$ for every $k \in \N$. 

Using Lemma \ref{thm-mixing} $(\rm i)$, we have $P' \cap M \subset (P \cap \rL(G))' \cap M \subset \rL(G)$ and so $P' \cap M = P' \cap \rL(G)$. Using \cite[Lemma 3.3 $(\rm v)$]{HI17}, denote by $z^\perp \in P' \cap \rL(G)$ the largest projection such that $P z^\perp \lessdot_M \rL(G)$. By contradiction, assume that $z^\perp \neq 1$ so that $z = (z^\perp)^\perp \neq 0$. We consider the following two possible situations that will give a contradiction.

Firstly, assume that for every $k \in \N$, we have $A_k z \npreceq_M \rL(G)$. Using Theorem \ref{appendix-thm-3}, we have $(A_kz)' \cap zMz \lessdot_M \rL(G)$ and so $Q_k z \lessdot_M \rL(G)$ since $Q_k z \subset (A_kz)' \cap zMz$ is with expectation. Since $Q_kz$ is globally invariant under $\sigma^\psi$, \cite[Remark 3.3]{Is17} implies that there exists a norm one projection $\Phi_k : z \langle M, \rL(G) \rangle z \to Q_k z$ such that $\Phi_k |_{zMz}$ is the unique $\psi$-preserving conditional expectation. Choose a nonprincipal ultrafilter $\omega \in \beta(\N) \setminus \N$ and define the completely positive map $\Phi_\omega : z \langle M, \rL(G) \rangle z \to P z$. Since $\Phi_\omega |_{zMz}$ is $\psi$-preserving, it follows that $\Phi_\omega$ is indeed a norm one projection such that $\Phi_\omega |_{zMz}$ is normal. Therefore, $Pz \lessdot_M \rL(G)$ and so $P = Pz^\perp \oplus Pz \lessdot_M \rL(G)$ by \cite[Lemma 3.3 $(\rm v)$]{HI17}. This contradicts the definition of $z^\perp$.

Secondly, assume that there exists $k \in \N$ such that $A_k z \preceq_M \rL(G)$. Since $A_k z$ is tracial, there exists $n \geq 1$, a normal $\ast$-homomorphism $\rho : A_k z \to \mathbf M_n(\rL(G))$ and a nonzero partial isometry $v \in \mathbf M_{1, n}(z M)\rho(z)$ such that $a v = v \rho(a)$ for every $a \in A_k z$. Lemma \ref{thm-mixing} $(\rm i)$ implies that $v^*v \in \mathbf M_n(\rL(G))$ and $v^* ((A_kz)' \cap zMz) v \subset (v^{*}v \rho(A_{k}z))' \cap v^*v\mathbf M_n(M)v^*v \subset v^*v\mathbf M_n(\rL(G))v^*v$. For every $\ell \geq k$, since $A_\ell z \subset A_k z$ is a diffuse subalgebra, Lemma \ref{thm-mixing} $(\rm i)$ implies that $v^* ((A_\ell z)' \cap zMz) v \subset (v^{*}v \rho(A_{\ell}z))' \cap v^*v\mathbf M_n(M)v^*v \subset v^*v\mathbf M_n(\rL(G))v^*v$. This shows that $v^* \, z\mathcal Pz \, v \subset v^*v \mathbf M_n(\rL(G)) v^*v$, where $\mathcal P = \bigvee_{k \in \N} ((A_k)' \cap M)$. Since $vv^* \in (A_k z)' \cap zMz\subset z\mathcal Pz$, we have $vv^* \, z \mathcal Pz \, vv^* \preceq_M \rL(G)$ and so $z\mathcal Pz \preceq_M \rL(G)$. Since $Pz \subset z \mathcal P z$ is with expectation, we have $Pz \preceq_M \rL(G)$ by \cite[Lemma 4.8]{HI15}. Using \cite[Proposition 4.10]{HI17}, there exists a nonzero projection $r \in \mathcal Z((Pz)' \cap z M z) \subset z(P' \cap M)z \subset z(P' \cap \rL(G))z$ such that $Pr \lessdot_M \rL(G)$. Therefore, $Pz^\perp \oplus Pr \lessdot_M \rL(G)$ by \cite[Lemma 3.3 $(\rm v)$]{HI17}. This contradicts the definition of $z^\perp$.

Therefore, we have $P \lessdot_M \rL(G)$. Applying item $(\rm ii)$ in Theorem \ref{thm-maximal}, we obtain $P \subset \rL(G)$.
\end{proof}

\begin{proof}[Proof of Application \ref{application-1}]
It follows from \cite[Corollaire 1.5.7]{Co72} that $\rT(M) = G$. Since $G < \R$ is a countable dense subgroup, \cite[Th\'eor\`eme 3.4.1]{Co72} implies that $M$ is a type ${\rm III_0}$ factor. Using Theorem \ref{thm-prime}, $M$ is a prime factor and using Theorem \ref{thm-cartan}, $M$ has no Cartan subalgebra.

Write $M = N \rtimes G$ and denote by $(u_g)_g$ the canonical unitaries in $M$ implementing the action $\sigma^\pi : G \curvearrowright N$. Denote by $\psi \in M_\ast$ the canonical faithful normal state. Observe that for every $g \in G$, we have $\sigma_g^\psi = \Ad(u_g)$. Let $\rL(G) \subset P \subset M$ be any intermediate von Neumann subalgebra. For every $g \in G$, we have $\sigma_g^\psi(P) = u_g P u_g^* = P$. Since $\sigma^\psi : \R \curvearrowright N$ is continuous with respect to the $u$-topology and since $G < \R$ is dense, we infer that $\sigma_t^\psi(P) = P$ for every $t \in \R$ and so $P \subset M$ is with expectation by \cite{Ta71}. The previous reasoning implies that any intermediate amenable subalgebra $\rL(G) \subset P \subset M$ is with expectation and Theorem \ref{letterthm-maximal} $(\rm i)$ implies that $\rL(G) = P$.
\end{proof}

\section{Strong solidity}

\subsection{Proof of Theorem \ref{letterthm-strongly-solid}}

\begin{proof}[Proof of Theorem \ref{letterthm-strongly-solid}]
We denote by $\varphi \in M_\ast$ the canonical faithful state. Assume that $G$ is amenable and $\pi : G \curvearrowright H_\R$ is a faithful mixing orthogonal representation. Without loss of generality, we may assume that $\ker U = \{0\}$ so that $M$ is a type ${\rm III_1}$ factor by Theorem \ref{letterthm-type}. Indeed, we may replace $H_\R$ by $H_\R \oplus (H_\R \otimes \rL^2_\R(\R))$, $U$ by $U \oplus (1_{H_\R} \otimes \lambda_\R)$ and $\pi$ by $\pi \oplus (\pi \otimes 1_{\rL^2_\R(\R)})$ so that 
$$M = \Gamma(U, \pi)\dpr \subset \Gamma(U \oplus (1_{H_\R} \otimes \lambda_\R), \pi \oplus (\pi \otimes 1_{\rL^2_{\R}(\R)}))\dpr = \mathfrak M.$$
Then $M \subset \mathfrak M$ is with expectation, $\pi \oplus (\pi \otimes 1_{\rL^2_{\R}(\R)})$ is mixing and  $\mathfrak M$ is a type ${\rm III_1}$ factor by Theorem \ref{letterthm-type}. Since solidity is preserved under taking diffuse subalgebras with expectation, up to replacing $M$ by $\mathfrak M$, we may assume that $M$ is a type ${\rm III_1}$ factor. By contradiction, assume that $M$ is not solid. Since $M$ is of type ${\rm III}$ and $M$ is not solid, there exists a diffuse abelian subalgebra with expectation $A \subset M$ such that the relative commutant $P = A' \cap M$ has no nonzero amenable direct summand. Theorem \ref{appendix-thm-3} implies that $A \preceq_M \rL(G)$. Then Theorem \ref{thm-mixing} $(\rm iii)$ implies that $P = A' \cap M \preceq_M \rL(G)$. Then \cite[Proposition 4.10]{HI17} implies that $P$ has a nonzero amenable direct summand. This is a contradiction.

We now prove that $M = \Gamma(U, \pi)\dpr$ is strongly solid following the proof of \cite[Main theorem]{BHV15}. We explain below the appropriate changes that are needed. As before, since strong solidity is preserved under taking diffuse subalgebras with expectation, we may assume that $M$ is a type ${\rm III_1}$ factor. By contradiction, assume that $M$ is not strongly solid. Since $M$ is a solid type ${\rm III}$ factor, the exact same reasoning as in the proof of \cite[Main theorem]{BHV15} shows that there exists a diffuse amenable subalgebra with expectation $Q \subset M$ such that $Q' \cap M = \mathcal Z(Q)$ and $P = \mathcal N_M(Q)\dpr$ has no nonzero amenable direct summand. Choose a faithful state $\phi \in M_\ast$ such that $Q \subset M$ is globally invariant under $\sigma^\phi$. Observe that $P$ is also globally invariant under $\sigma^\phi$. Denote by $(N, \psi)$ the unique Araki--Woods factor of type ${\rm III_1}$ endowed with any faithful normal state. We borrow notation from Appendix \ref{appendix-deformation}. Put $\mathcal M = \core_{\varphi \otimes \psi}(M \ovt N)$, $\mathcal B = \core_{\varphi \otimes \psi}(\rL(G) \ovt N)$, $\mathcal Q = \Pi_{\varphi \otimes \psi, \phi \otimes \psi}(\core_{\phi \otimes \psi}(Q \ovt N))$ and $\mathcal P = \Pi_{\varphi \otimes \psi, \phi \otimes \psi}(\core_{\phi \otimes \psi}(P \ovt N))$. By the claim in the proof of \cite[Main theorem]{BHV15}, we have $\mathcal P \subset \mathcal N_{\mathcal M}(\mathcal Q)\dpr$. Since this inclusion is with expectation and since $\mathcal P$ has no nonzero amenable direct summand by \cite[Proposition 2.8]{BHR12}, it follows that $\mathcal N_{\mathcal M}(\mathcal Q)\dpr$ has no nonzero amenable direct summand either.

Since the inclusion $\rL(G) \subset M$ is mixing, since $P = \mathcal N_M(Q)\dpr$ has no nonzero amenable direct summand and since $\rL(G)$ is amenable, a combination of Theorem \ref{thm-mixing} $(\rm iii)$ and \cite[Proposition 4.10]{HI17} implies that $Q \npreceq_M \rL(G)$. Since $Q \npreceq_M \rL(G)$ and since $Q' \cap M = \mathcal Z(Q)$,  \cite[Theorem A]{Is18} implies that $\mathcal Q \npreceq_{\mathcal M} \mathcal B$. Fix a nonzero finite trace projection $q \in \mathcal Q$ and observe that $q\mathcal Qq \npreceq_{\mathcal M} \mathcal B$. Since $M \ovt N$ is a type ${\rm III_1}$ factor, $\mathcal M$ is a type ${\rm II_\infty}$ factor. Since $\rL_{\varphi \otimes \psi}(\R) \subset \mathcal M$ is diffuse and with trace preserving conditional expectation, there exists a unitary $u \in \mathcal U(\mathcal M)$ such that $uqu^* \in \rL_{\varphi \otimes \psi}(\R) \subset \mathcal B$. Up to conjugating $\mathcal Q$ and $\mathcal P$ by $u \in \mathcal U(\mathcal M)$ we may assume that $q \in \mathcal B$ and that $q\mathcal Qq \npreceq_{\mathcal M} \mathcal B$. Working inside the type ${\rm II_1}$ factor $q \mathcal M q$, Theorem \ref{appendix-thm-1} implies that $(\theta_t)_t$ does not converge uniformly to the identity on $\Ball(q \mathcal Q q)$. Combining \cite[Theorem A]{HR10} and \cite[Lemma 4.6 and Theorem 4.9]{AD93}, the type ${\rm II_1}$ factor $q \mathcal M q$ has the complete metric approximation property. Moreover, since $\rL(G) \ovt N$ is amenable, the $q\mathcal M q$-$q\mathcal Mq$-bimodule $\rL^2(q \widetilde{\mathcal M} q) \ominus \rL^2(q\mathcal Mq)$ is weakly contained in the coarse bimodule $\rL^2(q\mathcal Mq) \otimes \rL^2(q\mathcal Mq)$ (see \cite[Section 4]{HR10}). Combining \cite[Proposition 3.7]{BHV15} with \cite[Proposition 2.4 and Corollary 2.5]{PV11}, we obtain that the stable normalizer $s\mathcal N_{q\mathcal Mq}(q \mathcal Qq)\dpr$ has a nonzero amenable direct summand. Since $q \mathcal P q \subset s\mathcal N_{q\mathcal Mq}(q \mathcal Qq)\dpr$, it follows that $\mathcal P$ has a nonzero amenable direct summand and so does $P \ovt N$ by \cite[Proposition 4.8]{BHR12}. This is a contradiction.
\end{proof}

\begin{proof}[Proof of Application \ref{application-2}]
Combining \cite[Theorem A]{HR10} and \cite[Consequence 4.10 (c)]{AD93}, $M$ has the complete metric approximation property. Likewise, combining \cite[Theorem 3.19]{HR10} and \cite[Corollary 5.15]{OT13}, $M$ has the Haagerup property. 

Since $\pi$ is faithful and mixing, Theorem \ref{letterthm-strongly-solid} implies that $M$ is a strongly solid factor. Let $(t, n) \in \R \times \Z$ be any element such that $U_t \otimes \pi_n = 1$. Since $\pi$ is mixing, we necessarily have $\pi_n = 1$. This implies that $n = 0$ and so $U_t = 1$. Thus, Theorem \ref{letterthm-type} shows that $\rT(M) = \ker(U)$.  Let $(t_k, n_k) \in \R \times \Z$ be any sequence such that $U_{t_k} \otimes \pi_{n_k} \to 1$ strongly as $k \to \infty$. Since $\pi$ is mixing, the sequence $(n_k)_k$ is necessarily bounded. This implies that $n_k = 0$ for $k \in \N$ large enough and so $U_{t_k} \to 1$ strongly as $k \to \infty$. Theorem \ref{letterthm-full} implies that $\tau(M) = \tau(U)$.

Denote by $\varphi \in M_\ast$ the canonical faithful state. Put $N = \Gamma(H_\R \otimes K_\R, U \otimes 1_{K_\R})\dpr$ so that $M = N \rtimes \Z$. Observe that $N_\varphi$ is a free group factor by \cite{Sh96} since $U \otimes 1_{K_\R}$ contains $1_\R \otimes 1_{K_\R}$ as a subrepresentation and $\dim(K_\R) = +\infty$. Observe that $M_\varphi = N_\varphi \rtimes \Z$. We claim that $M_\varphi$ is a factor. Indeed, since the action $\Z \curvearrowright N$ is mixing, we have $\mathcal Z(M_\varphi) \subset \rL(\Z)' \cap (N_\varphi \rtimes \Z) \subset \rL(\Z)$ by Theorem \ref{thm-mixing} $(\rm i)$. By contradiction, if $\mathcal Z(M_\varphi) \neq \C 1$, there exists $z \in \mathcal Z(M_\varphi) \subset \rL(\Z)$ such that $z \notin \C 1$. If we write $z = \sum_{n \in \Z} z_n u_n \in \rL(\Z)$ for its Fourier decomposition, there exists $n \in \Z \setminus \{0\}$ for which $z_n \in \C \setminus \{0\}$.  Let $x \in \mathcal U(N_\varphi)$ be any unitary such that $\varphi(x) = 0$. Since $x z = z x$, we have $ \sigma^\pi_n (x) = x$. For every $k \in \N$, we have $ \sigma^\pi_{k n} (x) = x$ and so $\varphi(\sigma^\pi_{k n}(x)x^*) = 1$. Since $\Z \curvearrowright N$ is mixing, we have $\lim_k \varphi(\sigma^\pi_{k n}(x)x^*) = \varphi(x)\varphi(x^*) = 0$. This is a contradiction and so $\mathcal Z(M_\varphi) = \C 1$. Thus, $M_\varphi$ is a nonamenable type ${\rm II_1}$ factor.

Finally, we show that $M$ is not isomorphic to any free Araki--Woods factor. By contradiction, assume that there exists a strongly continuous orthogonal representation $V : \R \curvearrowright L_\R$ such that $\Gamma (U \otimes 1_{K_\R}, 1_{H_\R} \otimes \pi)\dpr = M = \Gamma(L_\R, V)\dpr$. Denote by $\psi \in M_\ast$ the free quasi-free state state on $M = \Gamma(L_\R, V)\dpr$. Since $M_\varphi$ is nonamenable, \cite[Theorem 5.1]{HSV16} implies that there exists a nonzero partial isometry $v \in M$ such that $p = v^*v \in M_\varphi$, $q = vv^* \in M_\psi$ and $\Ad(v) : (pMp, \varphi_p) \to (qMq, \psi_q)$ is a state preserving isomorphism (here we simply denote $\varphi_p = \frac{\varphi(p \, \cdot \, p)}{\varphi(p)}$ and $\psi_q = \frac{\psi(q \, \cdot \, q)}{\psi(q)}$). In particular, we obtain 
$$p \, M_\varphi \, p = (pMp)_{\varphi_p} \cong (qMq)_{\psi_q} = q \,M_\psi \, q.$$ 
It follows that $q \, M_\psi \, q$ is nonamenable and so $M_\psi$ is a free group factor by \cite{Sh96}. This implies that $M_\varphi$ is an interpolated free group factor \cite{Dy92, Ra92}. By assumption, the $\rL(\Z)$-$\rL(\Z)$-bimodule $\rL^2(M \ominus \rL(\Z))$ is disjoint from the coarse $\rL(\Z)$-$\rL(\Z)$-bimodule $\ell^2(\Z \times \Z)$ (see \cite[Section 4]{HS09} for further details). This implies in particular that the $\rL(\Z)$-$\rL(\Z)$-bimodule $\rL^2(M_\varphi \ominus \rL(\Z))$ is also disjoint from the coarse $\rL(\Z)$-$\rL(\Z)$-bimodule $\ell^2(\Z \times \Z)$. Since $M_\varphi$ is an interpolated free group factor, this contradicts \cite[Corollary 7.6]{Vo95} and \cite[Proposition 9.2]{Sh97b}.
\end{proof}

\subsection{Semisolidity and absence of Cartan subalgebra}

Following \cite{Oz04, HI15}, we say that a $\sigma$-finite von Neumann algebra $M$ is {\em semisolid} if for any von Neumann subalgebra with expectation $Q \subset M$ that has no nonzero type ${\rm I}$ direct summand, the relative commutant $Q' \cap M$ is amenable. Any nonamenable semisolid factor $M$ is {\em prime}, that is, $M$ does not split as the tensor product of two diffuse factors. When the countable group $G$ is {\em virtually abelian}, we prove that the factor $\Gamma(U, \pi)\dpr$ is {\em semisolid} for any faithful orthogonal representation $\pi : G \curvearrowright H_\R$ such that $[U, \pi] = 0$. 

\begin{thm}\label{thm-prime}
Let $U : \R \curvearrowright H_\R$ be any strongly continuous orthogonal representation with $\dim H_\R \geq 2$. Let $G$ be any virtually abelian countable group and $\pi : G \curvearrowright H_\R$ any faithful orthogonal representation such that $[U, \pi] = 0$.

Then $\Gamma(U, \pi)\dpr$ is a semisolid factor. In particular, $\Gamma(U, \pi)\dpr$ is a prime factor.
\end{thm}

\begin{proof}
Put $M = \Gamma(U, \pi)\dpr$ and denote by $\varphi \in M_\ast$ the canonical faithful state. Since $\pi$ is faithful and $\dim H_\R \geq 2$, $M$ is a nonamenable factor by Theorem \ref{letterthm-type}. Let $Q \subset M$ be any von Neumann subalgebra with expectation such that $Q$ has no nonzero type ${\rm I}$ direct summand. Since $\rL(G)$ is of type ${\rm I}$, \cite[Lemma 2.6]{HV12} (or Lemma \ref{lem-intertwining}) implies that there exists a diffuse abelian subalgebra with expectation $A \subset Q$ such that $A \npreceq_M \rL(G)$. Theorem \ref{appendix-thm-3} implies that $A' \cap M$ is amenable. Since $Q' \cap M \subset A' \cap M$ is with expectation, it follows that $Q' \cap M$ is amenable. This shows that $M$ is semisolid. 
\end{proof}

When the countable group $G$ is abelian and $\pi$ is weakly mixing, we moreover prove that $\Gamma(U, \pi)\dpr$  has no Cartan subalgebra.

\begin{thm}\label{thm-cartan}
Let $U : \R \curvearrowright H_\R$ be any strongly continuous orthogonal representation with $\dim H_\R \geq 2$. Let $G$ be any abelian countable group and $\pi : G \curvearrowright H_\R$ any faithful weakly mixing orthogonal representation such that $[U, \pi] = 0$.

Then $\Gamma(U, \pi)\dpr$ has no Cartan subalgebra.
\end{thm}

\begin{proof}
Put $M = \Gamma(U, \pi)\dpr = N \rtimes G$ and denote by $\varphi \in M_\ast$ the canonical faithful state. Since $\pi$ is faithful and $\dim H_\R \geq 2$, $M$ is a nonamenable factor by Theorem \ref{letterthm-type}. By contradiction, assume that there exists a Cartan subalgebra $A \subset M$. If $A \npreceq_M \rL(G)$ and since $A \subset M$ is maximal abelian, we can argue exactly as in the last paragraph of the proof of Theorem \ref{letterthm-strongly-solid} and we obtain that $M = \mathcal N_M(A)\dpr$ is amenable. This is a contradiction. Therefore, we must have $A \preceq_M \rL(G)$. Since the orthogonal representation $\pi : G \curvearrowright H_\R$ is weakly mixing, the action $\sigma^\pi : G \curvearrowright (N, \varphi)$ is $\varphi$-weakly mixing. Theorem \ref{thm-mixing} $(\rm i)$ implies that $\mathcal{Q}\mathcal{N}_M(\rL(G))\dpr = \rL(G)$. In particular, $\rL(G) \subset M$ is maximal abelian and singular. Since $A \subset M$ and $\rL(G) \subset M$ are both masas with expectation, \cite[Theorem 2.5]{HV12} implies that there exists a partial isometry $w \in M$ such that $p = w^*w \in A$, $q = ww^* \in \rL(G)$ and $wAw^* = \rL(G)q$. By spatiality and using \cite[lemma 3.5]{Po03} (see also \cite[Proposition 2.3]{HU15b}), since $\mathcal N_M(A)\dpr = M$ and since $\mathcal N_M(\rL(G))\dpr = \rL(G)$, we obtain
$$qMq = w (\mathcal N_M(A)\dpr )w^*= \mathcal N_{qMq}(wAw^*)\dpr = \mathcal N_{qMq}(\rL(G)q)\dpr = q(\mathcal N_M(\rL(G))\dpr )q = \rL(G)q.$$
Thus, $qMq = \rL(G)q$ is amenable and so is $M$. This is a contradiction.
\end{proof}

\appendix

\section{Mixing inclusions}\label{appendix-mixing}

Let $(M, \varphi)$ be any $\sigma$-finite von Neumann algebra endowed with any faithful normal state. Let $B \subset M_\varphi$ be any von Neumann subalgebra. Denote by $\rE_B : M \to B$ the unique $\varphi$-preserving conditional expectation and  put $M \ominus B = \ker(\rE_B)$. Following \cite[Section 3]{Po03}, we say that the inclusion $B \subset M$ is 
\begin{itemize}
\item $\varphi$-{\em mixing} if for any uniformly bounded net $(b_k)_k$ in $B$ that converges weakly to $0$ as $k \to \infty$, we have
\begin{equation*}
\forall x, y \in M \ominus B, \quad \lim_k \|\rE_B(x^* b_k y)\|_\varphi = 0.
\end{equation*}

\item $\varphi$-{\em weakly mixing} if there exists a net of unitaries $(u_k)_k$ in $B$ such that
\begin{equation*}
\forall x, y \in M \ominus B, \quad \lim_k \|\rE_B(x^* u_k y)\|_\varphi = 0.
\end{equation*}
\end{itemize}

We recall the following important examples of $\varphi$-mixing (resp.\ $\varphi$-weakly mixing) inclusions (see \cite[Section 3]{Po03}). Let $G$ be any countable group, $(N, \varphi)$ any $\sigma$-finite von Neumann algebra endowed with any faithful normal state and $\sigma : G \curvearrowright (N, \varphi)$ any $\varphi$-preserving action. Put $M = N \rtimes G$. Denote by $\rE_N : M \to N$ the canonical faithful normal conditional expectation and still denote by $\varphi$ the faithful state $\varphi \circ \rE_N \in M_\ast$. Then $\rL(G) \subset M_\varphi$. We say that the action $\sigma : G \curvearrowright (N, \varphi)$ is 
\begin{itemize}
\item $\varphi$-{\em mixing} if we have
\begin{equation*}
\forall a, b \in N, \quad \lim_{g \to \infty} \varphi(\sigma_{g}(x)y) = \varphi(x) \varphi(y).
\end{equation*}

\item $\varphi$-{\em weakly mixing} if there exists a net $(g_k)_k$ in $G$ such that 
\begin{equation*}
\forall a, b \in N, \quad \lim_k \varphi(\sigma_{g_k}(x)y) = \varphi(x) \varphi(y).
\end{equation*}
\end{itemize}
By \cite[Section 3]{Po03}, if the action $\sigma : G \curvearrowright (N, \varphi)$ is $\varphi$-mixing (resp.\ $\varphi$-weakly mixing), then the inclusion $\rL(G) \subset M$ is $\varphi$-mixing (resp.\ $\varphi$-weakly mixing).

Following \cite[Section 3]{BHV15}, for any inclusion with expectation $Q \subset M$, we define the {\em stable normalizer} of $Q$ inside $M$ as the von Neumann subalgebra generated by the set
$$s\mathcal N_M(Q) = \left \{ x \in M \mid x^*Qx \subset Q \text{ and } xQx^* \subset Q \right \}.$$
Likewise, following \cite[Section 1]{Po01}, we define the {\em quasi normalizer} of $Q$ inside $M$ as the von Neumann subalgebra generated by the set
$$\mathcal{Q}\mathcal{N}_M(Q) = \left \{ x \in M \mid \exists x_1, \dots, x_k,  \; xQ \subset \sum_{j = 1}^k Qx_j \text{ and } Qx \subset \sum_{j = 1}^k x_j Q \right \}.$$
We have the following inclusions $Q \subset \mathcal N_M(Q)\dpr \subset s\mathcal N_M(Q)\dpr \subset \mathcal{Q}\mathcal{N}_M(Q)\dpr \subset M$ and they are all with expectation.

We prove technical properties of (weakly) mixing inclusions that generalize the main results of \cite[Section 3]{Po03} (see also \cite[Theorem D.4]{Va06} and \cite[Lemma 9.4]{Io12}). We should point out that compared to \cite[Section 3]{Po03}, the faithful normal state $\varphi$ is no longer assumed to be almost periodic and can be arbitrary.

\begin{lem}[{\cite[Lemma 2.3]{Po81}}]\label{lem-popa}
Let $(M, \varphi)$ be any $\sigma$-finite von Neumann algebra endowed with any faithful normal state. Let $Q \subset N \subset M_\varphi$ be any von Neumann subalgebras. Denote by $\rE_N : M \to N$ the unique $\varphi$-preserving conditional expectation and put $M \ominus N = \ker(\rE_N)$. Assume that $Q' \cap M = Q' \cap N$.

For every $x \in M \ominus N$, there exists $u \in \mathcal U(Q)$ such that $\|u x u^* - x\|_\varphi \geq \|x\|_\varphi$.
\end{lem}

\begin{proof}
The proof is exactly the same as the one of \cite[Lemma 2.3]{Po81} by averaging over $\mathcal U(Q)$ and exploiting the condition $Q' \cap M \subset N$. We give the details for the  reader's convenience. Let $x \in M \ominus N$ and define $\mathcal K_x \subset M$ as the weak closure in $M$ of the convex hull of the set $\{ux u^* \mid u \in \mathcal U(Q)\}$. Since $Q \subset M_\varphi$, the unique element $y \in \mathcal K$ of minimal $\|\cdot\|_\varphi$-norm satisfies $y \in Q' \cap M$ and so $y \in N$. Since $\rE_N (\mathcal K_x) = \{0\}$, we obtain $y = \rE_N(y) = 0$. 

By contradiction, assume that for every $u \in \mathcal U(Q)$, we have $\|u x u^* - x\|_\varphi < \|x\|_\varphi$. Then $x \neq 0$ and for every $u \in \mathcal U(Q)$, we have $\|x\|_\varphi^2 < 2 \Re (\varphi(x^* \, uxu^*))$. Taking weak limits of convex combinations of elements of the form $u x u^*$ for $u \in \mathcal U(Q)$, we obtain $\|x\|_\varphi^2 \leq 2 \Re(\varphi(x^*y)) = 0$. This is a contradiction.
\end{proof}

\begin{lem}\label{lem-mixing}
Let $(M, \varphi)$ be any $\sigma$-finite von Neumann algebra endowed with any faithful normal state. Let $B \subset M_\varphi$ be any von Neumann subalgebra. Let $1_P \in B$ be any nonzero projection and $P \subset 1_PB1_P$ any von Neumann subalgebra for which there exists a net of unitaries $(u_k)_k$ such that 
\begin{equation}\label{eq-mixing}
\forall x, y \in M \ominus B, \quad \lim_k \|\rE_B(x^* u_k y)\|_\varphi = 0.
\end{equation}
Let $\mathcal K \subset 1_P\rL^2(M)$ be any $P$-$B$-subbimodule that is of finite trace as a right $B$-module. Then $\mathcal K \subset 1_P\rL^2(B)$.
\end{lem}

\begin{proof}
By contradiction, assume that $\mathcal K \not\subset 1_P\rL^2(B)$. Using \cite[Lamma A.1]{Va06} and up to replacing $\mathcal K$ by $P_{\rL^2(M) \ominus \rL^2(B)}(\mathcal Kz)$ where $z \in \mathcal Z(B)$ is a large enough projection, we may assume that $\mathcal K \subset 1_P\rL^2(M) \ominus 1_P\rL^2(B)$ is a nonzero $P$-$B$-subbimodule that is finitely generated as a right $B$-module. Proceeding as in the proof of \cite[Theorem 2.3]{HV12}, there exist $n \geq 1$, a nonzero vector $\xi \in \mathbf M_{1,n}(\mathcal K)$ and a normal $\ast$-homomorphism $\pi : P \to \mathbf M_n(B)$ such that $a \xi = \xi \pi(a)$ for every $a \in P$. Denote by $\rE_{\mathbf M_n(B)} : \mathbf M_n(M) \to \mathbf M_n(B)$ the unique $(\varphi \otimes \tr_n)$-preserving conditional expectation and put $\mathbf M_n(M) \ominus \mathbf M_n(B) = \ker(\rE_{\mathbf M_n(B)})$. Letting $\xi^\ast = J_{\mathbf M_n(M)} (\xi) \in \rL^2(\mathbf M_n(M)) \ominus \rL^2(\mathbf M_n(B))$, we have $\xi^{*} \in \mathbf M_{n, 1}(J_{M} \mathcal K)$, $\xi^* 1_P = \xi^*$ and $\xi^* a = \pi(a) \xi^*$ for every $a \in P$. Write $\xi^* = v |\xi^*|$ for the polar decomposition of $\xi^*$ in the standard form of $\mathbf M_n(M)$. Then we have that $v \in \pi(1_P)\mathbf M_{n, 1}(M1_P)$ is a nonzero partial isometry such that $va = \pi(a) v$ for every $a \in P$. We moreover have that $|\xi^*| \in \rL^2(M)_+$ and $a |\xi^*| = |\xi^*| a$ for every $a \in P$. First, we prove the following claim. 
\begin{claim}\label{claim-commutant}
We have $v \in \pi(1_P)\mathbf M_{n, 1}(B1_P)$ and $P' \cap 1_PM1_P = P' \cap 1_PB1_P$.
\end{claim}

Indeed, for every $k \in \N$, we have
$$(v - \rE_{\mathbf M_n(B)}(v)) u_k = \pi(u_k) (v - \rE_{\mathbf M_n(B)}(v)).$$
Put $w = (v - \rE_{\mathbf M_n(B)}(v)) \in \mathbf M_n(M) \ominus \mathbf M_n(B)$. Since $\supp(ww^*) \leq \pi(1_P)$ and since $\pi(u_k) \in \mathcal U(\pi(1_P) \mathbf M_n(B) \pi(1_P))$ for every $k \in \N$, \eqref{eq-mixing} implies
$$\|\rE_{\mathbf M_n(B)}(w w^*) \|_{\varphi \otimes \tr_n} = \|\rE_{\mathbf M_n(B)}(\pi(u_k)w w^*) \|_{\varphi \otimes \tr_n}  = \|\rE_{\mathbf M_n(B)}(w u_k w^*) \|_{\varphi \otimes \tr_n} \to 0 \; \text{ as } k \to \infty.$$
Therefore, we have $w = 0$ and so $v = \rE_{\mathbf M_n(B)}(v) \in \pi(1_P)\mathbf M_{n, 1}(B1_P)$. Likewise, let $x \in P' \cap 1_PM1_P$ be any element. For every $k \in \N$, we have
$$(x - \rE_{B}(x)) u_k = u_k (x - \rE_{B}(x)).$$
Put $y = x - \rE_{B}(x) \in 1_PM1_P \ominus 1_PB1_P$. Since $\supp(yy^*) \leq 1_P$ and since $u_k \in \mathcal U(1_P B 1_P)$ for every $k \in \N$, \eqref{eq-mixing} implies
$$\|\rE_{B}(yy^*) \|_{\varphi} = \|\rE_{B}(u_kyy^*) \|_{\varphi}  = \|\rE_{B}(y u_k y^*) \|_{\varphi} \to 0 \; \text{ as } k \to \infty.$$
Therefore, we have $y = 0$ and so $x = \rE_{B}(x) \in B$. Next, we prove the following claim.

\begin{claim}\label{claim-contradiction}
We have $|\xi^*| \in \rL^2(B)$.
\end{claim}

Indeed, write $\eta$ for the orthogonal projection of $|\xi^*|$ onto $\rL^2(M) \ominus \rL^2(B)$. We show that $\eta = 0$. We still have $a \eta = \eta  a$ for every $a \in P$. Since $\eta \in \rL^2(M) \ominus \rL^2(B)$, since $1_P \eta 1_P = \eta$ and since $1_P \in M_\varphi$, we may choose a sequence $x_j \in M \ominus B$ such that $1_P x_j 1_P = x_j$ for every $j \in \N$ and $x_j \xi_\varphi \to \eta$ as $j \to \infty$. Note that $(x_j)_j$ need not be uniformly bounded. Since $P' \cap 1_PM1_P = P' \cap 1_PB1_P$ by Claim \ref{claim-commutant}, Lemma \ref{lem-popa} implies that for every $j \in \N$, there exists a unitary $u_j \in \mathcal U(P)$ such that $\|u_j x_j u_j^* - x_j\|_\varphi \geq \|x_j\|_\varphi$. Then for every $j \in \N$, we have
\begin{align*}
2 \|\eta  - x_j\xi_\varphi \| & \geq \|u_j (\eta - x_j\xi_\varphi) u_j^* - (\eta - x_j \xi_\varphi)\| \\
& = \|u_j \, x_j \xi_\varphi \, u_j^* - x_j\xi_\varphi\| \\
& = \|u_j x_j u_j^* - x_j\|_\varphi \\
& \geq \|x_j\|_\varphi.
\end{align*}
Since $\lim_j \|\eta  - x_j\xi_\varphi \| = 0$, we have $\lim_j \|x_j\|_\varphi = 0$ and so $\eta = 0$. This shows that $|\xi^*| \in \rL^2(B)$. 

Combining Claims \ref{claim-commutant} and \ref{claim-contradiction}, we obtain that $\xi^* = v |\xi^*| \in \rL^2(B)$. Since by construction $\xi^* \in \rL^2(M) \ominus \rL^2(B)$, we obtain $ \xi^* = 0$ and so $\xi = 0$. This is a contradiction.
\end{proof}

\begin{thm}\label{thm-mixing}
Let $(M, \varphi)$ be any $\sigma$-finite von Neumann algebra endowed with any faithful normal state. Let $B \subset M_\varphi$ be any von Neumann subalgebra. The following assertions hold:
\begin{itemize}
\item [$(\rm i)$] Assume that the inclusion $B \subset M$ is $\varphi$-weakly mixing. Then $\mathcal{Q}\mathcal{N}_{M}(B)\dpr =  B$.

\item [$(\rm ii)$] Assume that the inclusion $B \subset M$ is $\varphi$-mixing. Let $1_P \in B$ be any nonzero projection and $P \subset 1_PB1_P$ any diffuse von Neumann subalgebra. Then $\mathcal{Q}\mathcal{N}_{1_P M 1_P}(P)\dpr \subset 1_P B 1_P$.

\item [$(\rm iii)$] Assume that the inclusion $B \subset M$ is $\varphi$-mixing.  Let $1_Q \in M$ be any nonzero projection and $Q \subset 1_QM1_Q$ any diffuse von Neumann subalgebra with expectation. If $Q \preceq_M B$, then $\mathcal N_{1_QM1_Q}(Q)\dpr \preceq_M B$.
\end{itemize}
\end{thm}

\begin{proof} 
$(\rm i)$ Since the inclusion $B \subset M$ is $\varphi$-weakly mixing, there exists a net of unitaries $(u_k)_k$ in $B$ such that \eqref{eq-mixing} holds. Lemma \ref{lem-mixing} implies that $\mathcal{Q}\mathcal{N}_{M}(B)\dpr = B$.

$(\rm ii)$ Since $P$ is diffuse, we may choose a net of unitaries $(u_k)_k$ in $P$ that converges weakly to $0$ as $k \to \infty$. Since the inclusion $B \subset M$ is $\varphi$-mixing, \eqref{eq-mixing} holds for the net $(u_k)_k$. Lemma \ref{lem-mixing} implies that $\mathcal{Q}\mathcal{N}_{1_P M 1_P}(P)\dpr \subset 1_P B 1_P$.

$(\rm iii)$ Since $Q \preceq_M \rL(G)$, there exist projections $q \in Q$ and $p \in \rL(G)$, a unital normal $\ast$-homomorphism $\pi : qQq \to p\rL(G)p$ and a nonzero partial isometry $v \in pMq$ such that $a v = v \pi(a)$ for every $a \in qQq$. Note that $vv^* \in (qQq)' \cap qMq \subset \mathcal{Q}\mathcal{N}_{qMq}(qQq)\dpr$ and $v^*v \in \pi(qQq)' \cap p M p \subset p B p$ using item $(\rm ii)$. For every $x \in \mathcal{Q}\mathcal{N}_{qMq}(qQq)$, it is straightforward to see that $v^* x v \in \mathcal{Q}\mathcal{N}_{pMp}(\pi(qQq))$ and so $v^* \mathcal{Q}\mathcal{N}_{qMq}(qQq)\dpr  v \subset p B p$ using item $(\rm ii)$. This shows that $vv^* \mathcal{Q}\mathcal{N}_{qMq}(qQq)\dpr vv^* \preceq_M B$ and so $\mathcal{Q}\mathcal{N}_{qMq}(qQq)\dpr \preceq_M B$. Since $s\mathcal{N}_{qMq}(qQq)\dpr \subset \mathcal{Q}\mathcal{N}_{qMq}(qQq)\dpr$ is with expectation and since $q (s\mathcal{N}_{M}(Q)\dpr) q = s\mathcal{N}_{qMq}(qQq)\dpr$ (see \cite[Lemma 3.4]{BHV15}), this implies that $q (s\mathcal{N}_{M}(Q)\dpr) q\preceq_M B$ and so $s\mathcal{N}_{M}(Q)\dpr \preceq_M B$ (see \cite[Lemma 4.8]{HI15}). Since $\mathcal{N}_{M}(Q)\dpr \subset s\mathcal{N}_{M}(Q)\dpr$ is with expectation, we finally obtain $\mathcal{N}_{M}(Q)\dpr \preceq_M B$ by \cite[Lemma 4.8]{HI15}.
\end{proof}

\section{Popa's deformation/rigidity theory}\label{appendix-deformation}

\subsection*{Popa's malleable deformation}

Let $U : \R \curvearrowright H_\R$ be any strongly continuous orthogonal representation. Let $G$ be any countable group and $\pi : G \curvearrowright H_\R$ any orthogonal representation such that $[U, \pi] = 0$. Let $(N, \psi)$ be any $\sigma$-finite von Neumann algebra endowed with a faithful normal state. Define
\begin{itemize}
\item $M = \Gamma(U, \pi)\dpr$ and $\varphi$ the canonical faithful normal state on $\Gamma(U, \pi)\dpr$.
\item $\widetilde M = \Gamma(U \oplus U, \pi \oplus \pi)\dpr$ and $\widetilde \varphi$ the canonical faithful normal state on $\Gamma(U \oplus U, \pi \oplus \pi)\dpr$.
\item $\mathcal M = (M \ovt N) \rtimes_{\varphi \otimes \psi} \R$, the continuous core of $M \ovt N$ with respect to $\varphi \otimes \psi$.
\item $\widetilde{\mathcal M} = (\widetilde M \ovt N) \rtimes_{\widetilde \varphi \otimes \psi} \R$, the continuous core of $\widetilde M \ovt N$ with respect to  $\widetilde \varphi \otimes \psi$.
\item $\mathcal B = (\rL(G) \ovt N) \rtimes_{\varphi \otimes \psi} \R$, the continuous core of $\rL(G) \ovt N$ with respect to $\varphi \otimes \psi$.
\end{itemize}
We can regard $\widetilde{\mathcal M}$ as the semifinite amalgamated free product von Neumann algebra
\begin{equation*}
\widetilde{\mathcal M} = \left( (M \ovt N) \rtimes_{\varphi \otimes \psi} \R \right) \ast_{\mathcal B} \left( (M \ovt N) \rtimes_{\varphi \otimes \psi} \R\right),
\end{equation*}
where we identify $\mathcal M$ with the left copy of $(M \ovt N) \rtimes_{\varphi \otimes \psi} \R$ inside the amalgamated free product. We simply denote by $\tau$ the canonical faithful normal semifinite trace on $\mathcal M$ and by $\|\cdot\|_2$ the $2$-norm associated with $\tau$. Consider the following orthogonal transformations on $H_\R \oplus H_\R$:
$$
W  =  
\begin{pmatrix}
1 & 0 \\
0 & -1
\end{pmatrix}  \quad \text{ and } \quad
V_t  =  
\begin{pmatrix}
\cos(\frac{\pi}{2} t) & -\sin(\frac{\pi}{2} t) \\
\sin(\frac{\pi}{2} t) & \cos(\frac{\pi}{2} t)
\end{pmatrix}, \forall t \in \R.
$$
Define the associated state preserving deformation $(\theta_t, \beta)$ on $\widetilde M \ovt N$ by 
\begin{equation*}
\theta_t = \Ad(\mathcal F (U_t)) \otimes \id_N \quad \text{ and } \quad \beta = \Ad(\mathcal{F}(V)) \otimes \id_N. 
\end{equation*}
Since $V_t$ and $W$ commute with $\pi \oplus \pi$ and $U \oplus U$, it follows that $\theta_t$ and $\beta$ commute with the actions $(\sigma^\pi \ast \sigma^\pi) \otimes \id_N$ and $\sigma^{\widetilde \varphi \otimes \psi}$. We can then extend the deformation $(\theta_t, \beta)$ to $\widetilde{\mathcal M}$ after defining $\beta |_{\mathcal B} = \id_{\mathcal B}$ and $\theta_t|_{\mathcal B}  = \id_{\mathcal B}$ for every $t \in \R$. Moreover, it is easy to check that the deformation $(\theta_t, \beta)$ is {\em malleable} in the sense of Popa (see \cite{Po03}):
\begin{itemize}
\item [$(\rm i)$] $\lim_{t \to 0} \|x - \theta_t(x)\|_2 = 0$ for all $x \in \widetilde{\mathcal M} \cap \rL^2(\widetilde {\mathcal M})$.
\item [$(\rm ii)$] $\beta^2 = \id_{\widetilde{\mathcal M}}$ and $\theta_t \beta = \beta \theta_{-t}$ for all $t \in \R$.
\end{itemize}
Since $\theta_t , \beta \in \Aut(\widetilde{\mathcal M})$ are trace-preserving, we will also denote by $\theta_t, \beta \in \mathcal U(\rL^2(\widetilde{\mathcal M}))$ the corresponding Koopman unitary operators.

\subsection*{Locating subalgebras}

We can locate subalgebras of $\mathcal M$ for which the deformation $(\theta_t)_t$ converges uniformly to the identity with respect to $\|\cdot\|_2$ on the unit ball.

\begin{thm}[{\cite[Theorem 4.3]{HR10}, \cite[Theorem 2.10]{Ho12b}}]\label{appendix-thm-1}
Keep the same notation as above. Let $p \in \mathcal M$ be any nonzero finite trace projection and $\mathcal P \subset p \mathcal M p$ any von Neumann subalgebra. If the deformation $(\theta_t)_t$ converges uniformly in $\|\cdot\|_2$ on $\Ball(\mathcal P)$, then $\mathcal P \preceq_{\mathcal M} \mathcal B$.
\end{thm}

\begin{proof}
Since the deformation $(\theta_t)_t$ converges uniformly in $\|\cdot\|_2$ on $\Ball(\mathcal P)$, there exist $\kappa > 0$ and $n \in \N$ large enough so that $\tau(\theta_{2^{-n}}(u) u^*) \geq \kappa$ for all $u \in \mathcal U(\mathcal P)$. Now the rest of the proof is almost entirely identical to the one of \cite[Theorem 4.3]{HR10} except for some obvious modifications.

By contradiction, assume that $\mathcal P \npreceq_{\mathcal M} \mathcal B$ and choose a net of unitaries $u_k \in \mathcal U(\mathcal P)$ such that $\lim_k \|\rE_{\mathcal B}(b^* u_k a)\|_2 = 0$ for all $a, b \in p \mathcal M$. Using Popa's malleable deformation in combination with \cite[Theorem 2.5]{BHR12} (in lieu of \cite[Theorem 2.4]{CH08}), there exists a nonzero partial isometry $v \in p \widetilde{\mathcal M} \theta_1(p)$ such that $vv^* \in p \mathcal M p$, $v^*v \in \theta_1(p \mathcal M p)$ and $x v = v \theta_1(x)$ for every $x \in \mathcal P$. Using \cite[Claim in the proof of Theorem 3.3]{BHR12} (in lieu of \cite[Claim 4.4]{HR10}), we infer that 
$$\|v^*v\|_2 = \| \rE_{\theta_1(\mathcal M)}(v^*v) \|_2 = \lim_k \| \rE_{\theta_1(\mathcal M)}(v^*v \theta_1(u_k)) \|_2 = \lim_k \| \rE_{\theta_1(\mathcal M)}(v^* u_k v) \|_2 = 0.$$
This contradicts the fact that $v \neq 0$. Thus, we have $\mathcal P \preceq_M \mathcal B$.
\end{proof}

\subsection*{Popa's spectral gap rigidity}

We prove the following general spectral gap rigidity result inside $\mathcal M$.

\begin{thm}[{\cite[Theorem 6.5]{Ho12b}}]\label{appendix-thm-2}
Keep the same notation as above. Let $p \in \mathcal M$ be any nonzero finite trace projection and $\mathcal Q \subset p \mathcal M p$ any von Neumann subalgebra. Then at least one of the following assertions is true:
\begin{itemize}
\item There exists a nonzero projection $z \in \mathcal Z(\mathcal Q' \cap p\mathcal M p)$ such that $\mathcal Q z \lessdot_{\mathcal M} \mathcal B$.
\item The deformation $(\theta_t)_t$ converges uniformly in $\|\cdot\|_2$ on $\Ball(\mathcal Q' \cap p \mathcal M p)$.
\end{itemize}
\end{thm}

\begin{proof}
We follow the proof of \cite[Theorem A.1]{HI17}. Assume that the deformation $(\theta_t)_t$ does not converge uniformly in $\|\cdot\|_2$ on $\Ball(\mathcal Q' \cap p \mathcal M p)$. Then there exist $c > 0$, a sequence $(t_k)_k$ of positive reals such that $\lim_k t_k = 0$ and a sequence $(x_k)_k$ in $\Ball(\mathcal Q' \cap p \mathcal M p)$ such that $\|x_k - \theta_{2t_k}(x_k)\|_2 \geq c$ for all $k \in \N$.

Denote by $I$ the directed set of all pairs $(\varepsilon, \mathcal F)$ with $\varepsilon  > 0$ and $\mathcal F \subset \Ball(\mathcal Q)$ finite subset with order relation $\leq$ defined by
$$(\varepsilon_1, \mathcal F_1) \leq (\varepsilon_2, \mathcal F_2) \quad \text{if and only if} \quad \varepsilon_2 \leq \varepsilon_1, \mathcal F_1 \subset \mathcal F_2.$$
Let $i = (\varepsilon, \mathcal F) \in I$ and put $\delta = \min(\frac{\varepsilon}{4}, \frac{c}{8})$. Choose $k \in \N$ large enough so that $\|p - \theta_{t_k}(p)\|_2 \leq \delta$ and $\|a - \theta_{t_k}(a)\|_2 \leq \varepsilon/4$ for all $a \in \mathcal F$. 

Put $\xi_i = \theta_{t_k}(x_k) - \rE_{ \mathcal M }(\theta_{t_k}(x_k)) \in \rL^2( \widetilde{\mathcal M} ) \ominus \rL^2( \mathcal M )$ and $\eta_i = p \xi_i p \in \rL^2(p \widetilde{\mathcal M} p) \ominus \rL^2(p \mathcal M p)$. By the transversality property of the malleable deformation $(\theta_t)_{t}$ (see \cite[Lemma 2.1]{Po06}), we have 
$$
    \|\xi_i\|_2 \geq \frac{1}{2} \|x_k - \theta_{2t_k} (x_k)\|_2 \geq \frac{c}{2}.
$$
Observe that $\|p\theta_{t_k}(x_k)p - \theta_{t_k}(x_k)\|_2 \leq 2 \|p - \theta_{t_k}(p)\|_2 \leq 2 \delta$. Since $p \in \mathcal M$, by Pythagoras theorem, we moreover have  
$$\|p\theta_{t_k}(x_k)p - \theta_{t_k}(x_k)\|_2^2 = \|\rE_{ \mathcal M }(p\theta_{t_k}(x_k)p - \theta_{t_k}(x_k))\|_2^2 + \|\eta_i - \xi_i\|_2^2$$ 
and hence $\|\eta_i - \xi_i\|_2 \leq 2\delta$. This implies that 
$$\|\eta_i\|_2 \geq \|\xi_i\|_2 - \|\eta_i - \xi_i\|_2 \geq  \frac{c}{2} - 2 \delta \geq \frac{c}{4}.$$
For all $x \in p \mathcal M p$, we have  
$$\|x \eta_i\|_2 = \|(1 - \rE_{ \mathcal M}) (x \theta_{t_k} (x_k)p) \|_2 \leq \|x \theta_{t_k}(x_k)p\|_2 \leq \|x\|_2.$$
By Popa's spectral gap argument \cite{Po06}, for all $a \in \mathcal F \subset \Ball(\mathcal Q) \subset \Ball(p \mathcal M p)$, since $ax_k = x_k a$ for all $k \in \N$, we have  
  \begin{align*}
    \| a \eta_i - \eta_i a \|_2
    & =
    \|(1 - \rE_{\mathcal M}) (a \theta_{t_k}( x_k )p - p\theta_{t_k}( x_k) a)\|_2 \\
    & \leq \|a \theta_{t_k}( x_k)p - p\theta_{t_k}( x_k) a\|_2 \\
    & \leq 2 \|a - \theta_{t_k}(a)\|_2 + 2 \|p - \theta_{t_k}(p)\|_2 \\
    & \leq \frac{\varepsilon}{2} + \frac{\varepsilon}{2} = \varepsilon.
  \end{align*}
Thus $\eta_i \in \rL^2(p \widetilde {\mathcal M} p) \ominus \rL^2(p \mathcal M p)$ is a net of vectors satisfying $\limsup_i \|x \eta_i\|_2 \leq \|x\|_2$ for all $x \in p \mathcal M p$, $\liminf_i \|\eta_i\|_2 \geq~\frac{c}{4}$ and $\lim_i \|a \eta_i - \eta_i a\|_2 = 0$ for all $a \in \mathcal Q$. 

By construction of the amalgamated free product von Neumann algebra $\widetilde{\mathcal M} = \mathcal M \ast_{\mathcal B} \mathcal M$, there exists a $\mathcal B$-$\mathcal L$-bimodule such that we have $\rL^2(\widetilde {\mathcal M}) \ominus \rL^2(\mathcal M) \cong \rL^2(\mathcal M) \otimes_{\mathcal B} \mathcal L$ as $\mathcal M$-$\mathcal M$-bimodules (see e.g.\ \cite[Section 2]{Ue98}). The existence of the net $(\eta_i)_{i \in I}$ in combination with \cite[Lemma A.2]{HU15b} shows that there exists a nonzero projection $z \in \mathcal Z(\mathcal Q' \cap p \mathcal Mp)$ such that $\mathcal Q z \lessdot_{\mathcal M} \mathcal B$.
\end{proof}

We deduce the following spectral gap rigidity result inside $M \ovt N$.

\begin{thm}\label{appendix-thm-3}
Keep the same notation as above. Let $A \subset M \ovt N$ be any abelian von Neumann subalgebra with expectation. 

If $A \npreceq_{M \ovt N}( \rL(G) \ovt N)$, then $A' \cap (M \ovt N) \lessdot_{M \ovt N} (\rL(G) \ovt N)$. 
\end{thm}

\begin{proof}
Choose a faithful state $\phi \in (M \ovt N)_\ast$ such that $A \subset (M \ovt N)_\phi$. Observe that $Q =  A' \cap (M \ovt N)$ is globally invariant under $\sigma^\phi$. Assume that $Q$ is not amenable relative to $\rL(G) \ovt N$ inside $M \ovt N$. Put $\core(Q) = \Pi_{\varphi \otimes \psi, \phi}(\core_{\phi}(Q)) \subset \mathcal M$. Using \cite[Theorem 3.2]{Is17}, $\core(Q)$ is not amenable relative to $\mathcal B$ inside $\mathcal M$. Using \cite[Lemma 3.3]{HI17}, there exists a nonzero finite trace projection $q \in \rL_\phi(\R)$ such that $\Pi_{\varphi \otimes \psi, \phi}(q)\core(Q)\Pi_{\varphi \otimes \psi, \phi}(q)$ is not amenable relative to $\mathcal B$ inside $\mathcal M$. Using again \cite[Lemma 3.3]{HI17}, there exists a nonzero projection $p \in \mathcal Z((\Pi_{\varphi \otimes \psi, \phi}(q)\core(Q)\Pi_{\varphi \otimes \psi, \phi}(q))' \cap \Pi_{\varphi \otimes \psi, \phi}(q) \mathcal M \Pi_{\varphi \otimes \psi, \phi}(q))$ such that with $\mathcal Q = p\core(Q)p$, we have that $\mathcal Q z$ is not amenable relative to $\mathcal B$ inside $\mathcal M$ for any nonzero projection $z \in \mathcal Z(\mathcal Q' \cap p \mathcal Mp )$. Theorem \ref{appendix-thm-2} implies that the deformation $(\theta_t)_t$ converges uniformly in $\|\cdot\|_2$ on $\Ball(\mathcal Q' \cap p \mathcal M p)$. Since $\Pi_{\varphi \otimes \psi, \phi}(\pi_\phi(A)) p \subset \mathcal Q' \cap p \mathcal M p$ is a von Neumann subalgebra, the deformation $(\theta_t)_t$ converges uniformly in $\|\cdot\|_2$ on $\Ball(\Pi_{\varphi \otimes \psi, \phi}(\pi_\phi(A)) p)$. Theorem \ref{appendix-thm-1} implies that $\Pi_{\varphi \otimes \psi, \phi}(\pi_\phi(A)) p \preceq_{\mathcal M} \mathcal B$. Since $p = \Pi_{\varphi \otimes \psi, \phi}(q) p$, we have $\Pi_{\varphi \otimes \psi, \phi}(\pi_\phi(A)q)  \preceq_{\mathcal M} \mathcal B$. Then \cite[Lemma 2.4]{HU15a} implies that $A \preceq_{M \ovt N} (\rL(G) \ovt N)$.
\end{proof}

\bibliographystyle{plain}

\end{document}